\documentclass[runningheads]{siamonline220329}
\usepackage[T1]{fontenc}
\usepackage{graphicx}
\usepackage{amssymb,amsmath}
\usepackage{esint}
\usepackage{tikz}
\usepackage{pgfplots}
\usepackage{enumitem}
\usepgfplotslibrary{external}
\newcommand{\eqdef}{\stackrel{\rm def}{=}}
\newcommand{\om}{\Omega}

\newcommand{\R}{\mathbb{R}}
\renewcommand{\a}{\alpha}
\newcommand{\ep}{\varepsilon}

\newcommand{\cv}{{\mathbf c}}

\newcommand{\mL}{{\mathcal L}}

\newcommand{\be}{\begin{eqnarray}}
\newcommand{\ee}{\end{eqnarray}}

\newcommand{\av}{-\hspace{-.16in}\int}
\newcommand{\avsmall}{-\hspace{-.13in}\int}

\newcommand{\esssup}{{\rm ess\, sup}}

\renewcommand{\leq}{\leqslant}
\renewcommand{\geq}{\geqslant}

\newcommand{\supp}{{\rm supp}\,}

\newcommand{\half}{\frac{1}{2}}
\newcommand{\1}{{\mathbf 1}}

\newcommand{\weak}{\rightharpoonup}
\newcommand{\weakstar}{\stackrel{*}{\rightharpoonup}}

\newcommand\dx{\mathrm{d} x}

\newcommand\mA{\mathcal A}

\theoremstyle{remark}\newtheorem{remark}{Remark}

\begin{document}
\title{A nonlinear elasticity model in computer vision}
%
%
\author{John M. Ball
 \and
Christopher L. Horner
%
%
\thanks{ Heriot-Watt University and Maxwell Institute for the Mathematical Sciences, Edinburgh, U.K.}
\\
}

\maketitle              
\begin{abstract}

The purpose of this paper is to analyze a nonlinear elasticity model introduced by the authors in \cite{ballhorner} for comparing two images, regarded as bounded open subsets of $\R^n$ together with associated vector-valued intensity maps. Optimal transformations between the images are sought as minimisers of an integral functional among orientation-preserving homeomorphisms. The existence of minimisers is proved under natural coercivity and polyconvexity conditions, assuming only that the intensity functions are bounded measurable. Variants of the existence theorem are also proved, first under the constraint that  finite sets of landmark points in the two images are mapped one to the other, and second when one image is to be compared to an unknown part of another.

The question is studied as to whether for  images related by an affine mapping the unique minimiser is given by that affine mapping.  For a natural class of functional integrands an example is given guaranteeing that this property holds for pairs of images in which the second is a scaling of the first by a constant factor. However for the property to hold for arbitrary pairs of affinely related images it is shown that the integrand has to depend on the gradient of the transformation as a convex function of its determinant alone. This suggests a new model in which the integrand  depends also on second derivatives of the transformation, and an example is given for which both existence of minimisers is assured and the above property holds for all pairs of affinely related images. 
\end{abstract}
\section{Introduction}
The purpose of this paper is to  analyze a nonlinear elasticity model introduced in \cite{ballhorner} for comparing two images $P_1=(\Omega_1,c_1),\; P_2=(\Omega_2,c_2)$, regarded as bounded Lipschitz domains $\om_1,\om_2$ in $\R^n, n\geq 1$, with corresponding measurable intensity maps $c_1:\om_1\to K, c_2:\om_2\to K$, where $K\subset\R^m$ is bounded, closed and convex. The model is based on an integral functional
\be
\label{0} I_{P_1,P_2}(y)=\int_{\Omega_1}\psi(c_1(x), c_2(y(x)),Dy(x)) \,\dx,
\ee
 depending on $c_1,c_2$ and  a map $y:\om_1\to\om_2$ with  gradient $Dy$, whose minimisers give  optimal transformations  between images. The admissible transformations $y$ between the images are assumed to be orientation-preserving homeomorphisms with $y(\om_1)=\om_2$, and are not required to satisfy other boundary conditions. The integrand $\psi:K\times K\times M^{n\times n}_+\to[0,\infty)$ is assumed to be continuous. Here $M^{n\times n}_+:=\{A\in M^{n\times n}:\det A>0\}$, where $M^{n\times n}$ denotes the set of real $n\times n$ matrices. We study in particular whether $\psi$ can be chosen so that when $P_2$ is  related to $P_1$ by an affine mapping the unique minimiser of 
$I_{P_1,P_2}$ is given by that affine mapping.
 
There are various different approaches to image comparison, for example using optimal transport \cite{solomon2015convolutional},  flow of diffeomorphisms (metamorphosis) \cite{younes} and machine learning. Approaches based on linear elasticity are also often used (see, for example, \cite{modersitzki}). The use of nonlinear elasticity is less common (a fairly complete list of papers is \cite{droskerumpf04,rumpf2013,rumpfwirth,burgeretal2013,debrouxetal20,iglesiasrumpfscherzer,	iglesias21,linetal2010,OzereLeguyader2015,OzereLeguyader2015a,simonetal2017,debroux}) but offers significant advantages over linear elasticity because it allows for large deformations and respects rotational invariance.

Our model is closely related to those of Droske \& Rumpf \cite{droskerumpf04}, Iglesias \cite{iglesias21}, Rumpf \cite{rumpf2013} and Rumpf \& Wirth \cite{rumpfwirth}, and like them (as also done in  \cite{burgeretal2013,debrouxetal20,iglesiasrumpfscherzer}) we adapt the existence theory for polyconvex energies in nonlinear elasticity in \cite{j8} to our situation.  

Our model is briefly reviewed in Section \ref{nle}, in which invariance conditions on $\psi$ from \cite{ballhorner} are recalled. In Section \ref{existence} the existence of an absolute minimiser for \eqref{0} is established (see Theorem \ref{exthm}) under continuity, polyconvexity and coercivity conditions on the integrand $\psi$, assuming only that the intensity functions $c_1,c_2$ are in $L^\infty$. The coercivity condition is weaker than that assumed in \cite{ballhorner} in that, making use of a recent result \cite{ballonnineniwaniec},  we can allow $n^{\rm th}$ power growth of $\psi$ rather than $p^{\rm th}$ power growth for $p>n$. 

Two variants of the existence theorem are given. In the first  (Theorem \ref{thm:landmark}) existence of a minimiser is proved for $n\geq 2$ under the constraint that a finite number of distinct landmark points in $\om_1$ are mapped to given distinct points in $\om_2$. The case when some of the landmark points belong to $\partial\om_1$ and $\partial\om_2$ is discussed. In the second (Theorem \ref{thm:closedexistence})  existence of a minimiser is proved in the set of orientation-preserving homeomorphisms $y:\om_1\to a+A\om_1\subset\om_2$, where $a\in\R^n$ and $A$ belongs to a relatively closed subset of $M^{n\times n}_+$, corresponding to matching a template image $P_1=(\om_1,c_1)$ with part of $P_2=(\om_2,c_2)$  allowing for changes of scale and orientation.

In Section \ref{metricsection}, motivated by the work of Berkels, Effland and Rumpf \cite{BerkelsEfflandRumpf2015}  and   the metamorphosis model of Miller and Younes \cite{MillerYounes}, Trouv\'e and Younes \cite{TrouveYounes}, \cite{TrouveYounesa}, we discuss connections of the minimization problem with image morphing and possible metrics between images having  the same domain $\om=\om_1=\om_2$. In particular we show the existence of a `discrete morphing sequence' between such a pair of images for suitable integrands of the mass-conserving form
\be
\label{masscon}
\psi(c_1,c_2,A)=\Psi(A)+(1+(\det A)^{-1})|c_1- c_2\det A|^2,
\ee
(in which the intensity of magnified images diminishes in proportion to the Jacobian of the deformation gradient) without the need for additional higher-order derivative terms as in \cite{BerkelsEfflandRumpf2015}.

In Section \ref{scaling} we consider the case when two images $P_1=(\om_1,c_1), P_2=(\om_2,c_2)$ are related by an affine transformation $y(x)=Mx+a$, where $M\in M^{n\times n}_+$ and $a\in\R^n$, so that
\be
\label{affineform}
\om_2=M\om_1+a,\;c_2(Mx+a)=c_1(x),
\ee
and ask for which $\psi$ the minimization algorithm delivers this affine transformation as the unique minimiser. We consider $\psi$  of the form
\be
\label{psiforma}
\psi(c_1,c_2,A)=\Psi(A)+(1+\det A)g(c_1,c_2),
\ee
where  $g:\R^m\times\R^m\to[0,\infty)$ is continuous with $g(c_1,c_2)=g(c_2,c_1)$ and $g(c_1,c_2)=0$ if and only if $c_1=c_2$
 We first show that $\Psi$ can be chosen so that $\psi$ satisfies both the invariance conditions in Section \ref{nle} and the hypotheses of Theorem \ref{exthm}, and  such that, under a mild nondegeneracy condition, for any pair of images $(P_1,P_2)$ related by a uniform magnification  the unique minimiser of $I_{P_1,P_2}$ is given by that magnification.  
However, for the functional to deliver as a minimiser the affine transformation between {\it any} affinely related pair of images,  we show (see Theorem \ref{generalM}) that $\Psi(A)=h(\det A)$ for some convex $h$. Theorem \ref{generalM} improves the corresponding statement in \cite{ballhorner} by assuming only that $\Psi$ is continuous and by weakening the regularity requirements on the boundary, and is proved by constructing suitable tangential variations. 

In the mass-conserving case \eqref{masscon} a different scaling question is relevant, whether $\Psi$ can be chosen such that  for all pairs of images $P_1=(\om_1,c_1), P_2=(\om_2,c_2)$ satisfying
\be
\label{masscscaling}
\om_2=M\om_1+a,\; c_2(Mx+a)\det M=c_1(x),
\ee
the unique minimiser is $y(x)=Mx+a$, and we show that the answers are the same as for \eqref{psiforma}.

Theorem \ref{generalM}  implies in particular that in order for the minimisation algorithm to deliver the affine transformation between affinely related images, $\psi$ cannot be coercive, rendering existence problematic. However by adding a suitable dependence on the second derivatives $D^2y$ we can retain this property as well as existence, and this is proved for an example in Theorem \ref{Dtwo}.

We end the paper with some brief remarks in Section \ref{regularitya} on the regularity of minimisers and possible numerical implementations.

\section{Nonlinear elasticity model}
\label{nle}
We represent an image by a pair $P=(\Omega,c)$, where $\Omega\subset\R^n$ is a bounded Lipschitz domain, and $c\in L^\infty(\om,K):=\{d\in L^\infty(\om,\R^m):d(x)\in K \text{ for a.e. }x\in\om\}$, $K$ a nonempty bounded closed convex subset of $\R^m$ (for example $K=[0,1]^m$),  is an {\it intensity map} describing the greyscale intensity ($m=1$), the intensity of colour channels ($m>1$), and possibly other image characteristics.  We seek to  compare two images 
$P_1=(\Omega_1,c_1),\; P_2=(\Omega_2,c_2)$ 
by minimising \eqref{0}  among a suitable class of orientation-preserving homeomorphisms $y:\om_1\to\om_2$.  Apart from the requirement that $y(\Omega_1)=\Omega_2$ we do not impose any other boundary condition, with the aim of better comparing images with important boundary features. 

We  identify images related by a proper rigid transformation through an equivalence relation; for two images $P=(\Omega,c)$ and $P'=(\Omega',c')$ we write $P\sim P'$ if  for some $a\in\R^n$, $R\in SO(n)$
$$\Omega'=E(\Omega),\; c'(E(x))=c(x), \text{ where }E(x)=a+Rx.$$

Natural conditions on the integrand $\psi$ are:\\

\noindent (C1) ({\it Invariance under proper rigid transformations})\\
\mbox{         }\hspace{.5in}If $P_1\sim P_1'$, $P_2\sim P_2'$, with corresponding rigid transformations $E_1(x)=a_1+R_1x$, $E_2(x)=a_2 +R_2x$,
then
\be
\label{1a} 
I_{P_1,P_2}(y)=I_{P_1',P_2'}(E_2\circ y\circ E_1^{-1}).
\ee

 It is shown in \cite{ballhorner} that (C1) is equivalent to the condition that $\psi(c_1,c_2,\cdot)$ is {\it isotropic}, i.e.
\be
\label{3}
 \psi(c_1,c_2,QAR)=\psi(c_1,c_2,A)
 \ee
   for all  $c_1,c_2\in K, A\in M^{n\times n}_+$, and $Q,R \in SO(n)$. (The range $K$ of the intensity maps was not made explicit in \cite{ballhorner}, but this does not affect the arguments.)\\

\noindent (C2)  ({\it Identification  of equivalent images.})\\
\mbox{         }\hspace{.5in}$I_{P_1,P_2}(y)=0$  iff $P_1\sim P_2 $ with corresponding rigid transformation $y$.\vspace{.1in}
 
It is shown in \cite{ballhorner} that, for sufficiently regular $y$,
(C2)  is equivalent to the condition that for $c_1,c_2\in K$
\be
\label{5}
\psi(c_1,c_2,A)=0 \text{ iff }c_1=c_2\text{ and }A\in SO(n).
\ee

\noindent(C3) ({\it Symmetry with respect to interchanging images.}) (See \cite{cachierrey2000,kolourietal2015, iglesias21,iwanieconninen2009}.)
$$I_{P_1,P_2}(y)=I_{P_2,P_1}(y^{-1}).$$ 

This holds for all $P_1,P_2$ (see \cite{ballhorner}) if and only if for $c_1,c_2\in K$
\be
\label{8}\psi(c_1,c_2,A)=\psi(c_2,c_1,A^{-1})\det A.
\ee
 
A class of integrands $\psi$ satisfying (C1)-(C3) is given by
\be
\label{8aa}
\psi(c_1,c_2,A)=\Psi(A)+f(c_1,c_2,\det A),
\ee
where\\ (a) $\Psi:M^{n\times n}_+\to [0,\infty)$ is continuous and isotropic, with $\Psi(A)=\det A \cdot\Psi(A^{-1})$, $\Psi^{-1}(0)=SO(n)$, \\(b) $f:K\times K\times (0,\infty)\to[0,\infty)$ is continuous, with $f(c_1,c_2,\delta)=\delta f(c_2,c_1,\delta^{-1})$, $f(c_1,c_2,1)=0$ iff $c_1=c_2$.

Examples of $f$ satisfying (b) are given by
\be 
\label{8a}f(c_1,c_2,\delta)=(1+\delta)g(c_1,c_2),
\ee
where $g:\R^m\times\R^m\to[0,\infty)$ is continuous with $g(c_1,c_2)=g(c_2,c_1)$ and $g(c_1,c_2)=0$ if and only if $c_1=c_2$, for example $g(c_1,c_2)=|c_1-c_2|^2$, 
and the mass-preserving form (correcting the corresponding expression in \cite{ballhorner})
\be
\label{8b}
f(c_1,c_2,\delta)= (1+\delta^{-1})|c_1-c_2\delta|^2,
\ee
which vanishes for $c_1=c_2\delta$ (see \cite{burgeretal2013}). 
\begin{remark}
\label{convexity}\rm Both forms of $f$ are convex in $\delta>0$, while $f$ given by \eqref{8b} is convex in the triple  $(\delta, c_1,c_2\delta)$ on account of the well-known result that $(\delta, c)\mapsto \frac{|c|^2}{\delta}$ is convex on $(0,\infty)\times\R^m$  (we use this observation in Theorem \ref{discrete} below). 
\end{remark}
Although the conditions (C1)-(C3) are natural and are used in Section \ref{scaling}, they do not form part of the hypotheses for the existence of minimisers in Theorem \ref{exthm} below.\vspace{1in}

\section{Existence of minimisers}
\label{existence}
\subsection{Problem formulation and hypotheses}Throughout we assume that $\om_1,\om_2\subset\R^n$ are bounded homeomorphic Lipschitz domains. 
  For $1\leq p\leq\infty$ define the set of admissible maps ${\mathcal A}_p={\mathcal A}_p(\om_1,\om_2)$ by
\be
{\mathcal A}_p=\{y\in W^{1,p}(\Omega_1,\R^n):\, y:\Omega_1\to \Omega_2 \text{ an orientation-preserving}&\\ &\hspace{-1.6in}\text{homeomorphism},  y^{-1}\in W^{1,p}(\Omega_2,\R^n)\}.\nonumber
\ee
Here, for a bounded domain $\om\subset\R^n$ and $1\leq p<\infty$, $W^{1,p}(\Omega,\R^n)$ is the Sobolev space of maps $y:\om\to\R^n$ that are measurable with respect to $n-$dimensional Lebesgue measure $\mL^n$ and possess a weak derivative $Dy$ such that
$$\|y\|_{1,p}:=\left(\int_\om\left(|y(x)|^p+|Dy(x)|^p\right)\,\dx\right)^\frac{1}{p}<\infty,
$$
and $W^{1,\infty}(\om,\R^n)$ is the set of such maps with $$\|y\|_{1,\infty}:=\esssup_{x\in\om}\left(|y(x)|+|Dy(x)|\right)<\infty.$$
We say that a map $y\in W^{1,p}(\om,\R^n)$ is {\it orientation-preserving} if  $\det Dy(x)>0$ for a.e. $x\in\om$. By a result of Vodop'yanov \& Gol'dstein \cite{vodopyanov79} any $y\in W^{1,n}(\om,\R^n)$ has a representative (which we always choose) that is continuous on $\om$. Furthermore (see e.g. Fonseca \& Gangbo \cite{fonsecagangbo95a})
  $ y$ maps sets of $\mL^n$ measure zero to sets of measure zero (the $N$-{\rm property}) and hence measurable sets to measurable sets, for  any $\mL^n$ measurable subset $E\subset \om$
the {\it multiplicity function} $N(y,\om,z):=\#\{x\in\om:y(x)=z\}$ is a measurable function of $z\in\R^n$, and for any
 $L^\infty$ map $\varphi:\R^n\to\R^s$ the change of variables formula
\be
\label{change}
\int_E\varphi(y(x))\det Dy(x)\,\dx=\int_{R^n}\varphi(z)N( y,E,z)\,{\mathrm d}z
\ee 
holds.  We will make use of the following result  \cite{ballonnineniwaniec}, which is motivated by earlier work of Iwaniec \& Onninen \cite{iwanieconninen2011a}:
\begin{theorem}
\label{W1n}
Let  $y^{(j)}:\om_1\to \om_2$ be a family of orientation-preserving homeomorphisms such that
\\ {\rm (i)} $y^{(j)}\weak y$ in $W^{1,n}(\om_1,\R^n)$,\\
{\rm (ii)} $\sup_j\int_{\om_1}h(\det Dy^{(j)}(x))\,dx<\infty$ for some continuous function $h:(0,\infty)\to[0,\infty)$ with $\lim_{\delta\to \infty}\frac{h(\delta)}{\delta}=\infty$.\\
Then  $y^{(j)}, y$ have continuous extensions $\tilde y^{(j)}, \tilde y$ to  $\bar\om_1$, and  $\tilde y^{(j)}\to \tilde y$ uniformly in $\bar \om_1$.
\end{theorem}
\begin{corollary}
\label{Apequiv}
For $p\geq n$ every $y\in \mA_p$ has an extension $\tilde y$ to $\bar\om_1$ with $\tilde y:\bar\om_1\to\bar\om_2$ surjective and a homeomorphism, and $\tilde y(\partial\om_1)=\partial\om_2$.
\end{corollary}
\begin{proof}
It suffices to consider the case $p=n$, so let $y\in \mA_n$. Then, since $\det A$  is an $n^{\rm th}$ order polynomial in $A$, $\det Dy\in L^1(\om_1)$. Hence there exists a continuous function 	 $h:(0,\infty)\to(0,\infty)$ such that $\lim_{\delta\to \infty}\frac{h(\delta)}{\delta}=\infty$ and $\int_{\om_1} h(\det Dy(x))\,dx<\infty$ (directly, or as a consequence of the de la Vall\'ee Poussin Theorem \cite[p24]{DellacherieMeyer}). Hence, applying Theorem \ref{W1n} to the constant sequence $y^{(j)}:=y$ we deduce that $y$ has a  extension $\tilde y:\bar\om_1\to \om_2$. Similarly $\xi:=y^{-1}$ has a continuous extension $\tilde\xi:\bar\om_2\to\bar\om_1$. Since 
\begin{align*}
\tilde\xi(\tilde y(x))=x,&\;x\in\om_1,\\\tilde y(\xi(z))=z,&\;z\in\om_2,
\end{align*}
we obtain the same relations for $x\in\bar\om_1$ and $z\in\bar\om_2$, from which the injectivity and surjectivity of $\tilde y$ and $\tilde\xi$ follows. Since $\tilde y(\om_1)=y(\om_1)=\om_2$ it follows that $\tilde y(\partial\om_1)=\partial\om_2$.
\end{proof}
In the sequel we drop the tildes and write $\tilde y=y,\,\tilde\xi=\xi$ etc.\\

We make the following technical hypotheses on $\psi$ and $c_1,c_2$.\vspace{.05in}\\
(H1) ({\it Continuity})  $\psi:K\times K\times M^{n\times n}_+\to [0,\infty)$ is continuous,\vspace{.05in}\\
(H2) ({\it Coercivity})  $\psi(c,d,A)\geq C(|A|^n+|A^{-1}|^n\det A+h(\det A)) $\\ for all $c,d\in K, A\in M^{n\times n}_+$, where $C>0$ is a constant and $h:(0,\infty)\to\R$ satisfies $\lim_{\delta\to \infty}\frac{h(\delta)}{\delta}=\infty, \lim_{\delta\to 0+}h(\delta)=\infty$,\vspace{.05in}\\
(H3) ({\it Polyconvexity}) $\psi(c,d,\cdot)$ is polyconvex for each $c,d\in K$, i.e. there is a function $g:K\times K\times \R^{\sigma(n)}\times (0,\infty)\to \R$ with $g(c,d,\cdot)$ convex, such that  
$$\psi(c,d,A)=g(c,d,{\mathbf J}_{n-1}(A), \det A)\text{ for all }c,d\in K, A\in M^{n\times n}_+,$$ where ${\mathbf J}_{n-1}(A)$ is the list of all minors (i.e. subdeterminants) of $A$ of order $\leq n-1$ and $\sigma(n)$ is the number of such minors,\vspace{.05in}\\
(H4) ({\it Bounded measurable intensities})  $c_1\in L^\infty(\Omega_1,K)$, $c_2\in L^\infty(\Omega_2,K)$.\vspace{.05in}

We note that a sufficient condition for (H2) to hold is that for some $p>n$
\be
\label{pnotn}
\psi(c,d,A)\geq C(|A|^p+|A^{-1}|^p\det A)-C_1,
\ee
where $C>0,C_1$ are constants, since then 
\be
\label{blowup}\psi(c,d,A)\geq \frac{C}{2}\left(|A|^n+|A^{-1}|^n\det A+h(\det A)\right),
\ee
where
\be
h(\delta):=n^{\frac{p}{2}}(\delta^{\frac{p}{n}}+ \delta^{1-\frac{p}{n}})-\delta-1-\frac{2C_1}{C}.
\ee
 This follows from the elementary inequality $t^p\geq t^n-1$ for $t\in(0,\infty)$ and the Hadamard inequality $|B|^n\geq n^\frac{n}{2}\det B$ for $B\in M^{n\times n}_+$ applied to $B=A$ and $B=A^{-1}$.

Examples of $\psi$ satisfying (H1)-(H3) as well as conditions (C1)-(C3) are given by \eqref{8aa} with $\Psi$ given by \eqref{isoPsi} below and $f$ given by \eqref{8a} or \eqref{8b}. Note that the polyconvexity condition (H3) is invariant under the interchange of images transformation \eqref{8} (see \cite[Theorem 2.6]{p4} for the case $n=3$).
\subsection{Existence theorem}
\begin{theorem}
\label{exthm}Suppose that  {\rm (H1)-(H4)} hold and that $\mathcal A_n$ is nonempty. Then there exists an absolute minimiser of
$$I_{P_1,P_2}(y)=\int_{\Omega_1}\psi(c_1(x), c_2(y(x)),Dy(x)) \,\dx$$
in $\mathcal A_n$.
\end{theorem}

Before proving Theorem \ref{exthm} we indicate some differences between our hypotheses and those in other papers where existence for nonlinear elasticity based image comparison functionals is proved without the addition of higher-order gradient terms in the integrand. Rumpf \cite{rumpf2013} takes $\om_1=\om_2$ with boundary condition $y(x)=x$ for $x\in\partial\om_1$ (invertibility of $y$ being assured via the global inverse function theorem \cite{j16}), and assumes both a more special additive polyconvex form for $\psi$ and the stronger condition that the intensity maps are continuous outside a singular set. Similar hypotheses are made in Droske \& Rumpf \cite{droskerumpf04} though the form of the registration term is different, and in Iglesias \cite{iglesias21}. In these papers stronger coercivity conditions of the type  \eqref{pnotn} with $p>n$ are assumed. The coercivity condition in \cite{burgeretal2013} involves minors in the spirit of \cite{j8}, but global invertibility is not addressed. Unlike in these papers we carry out a direct minimization in the class of Sobolev homeomorphisms without fixing $y$ on $\partial\om_1$, and avoiding the use of inverse function theorems.

\begin{proof}[Proof of Theorem \ref{exthm}] We suitably adapt the proofs of \cite[Theorem 4.1]{j17}, \cite[Theorem 6.1]{j26}. 
We first note that $I_{P_1,P_2}(y)$ is well defined for $y\in \mA_n$ since    $y^{-1}$ has the $N$-property and hence the function $c_2(y(\cdot))$ is measurable and independent of the representative of $c_2\in L^\infty(\om_1,\R^s)$. Also, by \eqref{change} we have that
\be
\label{inverseform}
I_{P_1,P_2}(y)=\int_{\om_2}\psi(c_1(\xi(z)),c_2(z), D\xi(z)^{-1})\det D\xi(z)\,{\mathrm d}z,
\ee
where $\xi:=y^{-1}$. 

 We may assume that $\inf_{\mA_n}I_{P_1,P_2}(y)<\infty$, since otherwise any $y\in\mA_n$ is a minimiser. 
Let $y^{(j)}$ be a minimizing sequence for $I_{P_1,P_2}(y)$ in $\mathcal A_n$ with corresponding inverses $\xi^{(j)}$. Then by (H2) and the boundedness of $\om_1$, $\om_2$ we have that $y^{(j)}$ is bounded in $W^{1,n}(\om_1,\R^n)$ and $\xi^{(j)}$ is bounded in $W^{1,n}(\om_2,\R^n)$. Also, if $n=1$, from the term in $h$ in (H2) and the de la Vall\'ee Poussin theorem, we have that $\det Dy^{(j)}=y^{(j)}_x$ and $\det D\xi^{(j)}=\xi^{(j)}_z$ are sequentially weakly compact in $L^1(\om_1)$, $L^1(\om_2)$ respectively. Thus we can assume that 
\begin{align*}
y^{(j)}\weak y\text{ in }W^{1,n}(\om_1,\R^n),\;
\xi^{(j)}\weak \xi\text{ in } W^{1,n}(\om_2,\R^n).
\end{align*}
By Corollary \ref{Apequiv} we have that
\begin{align*}
\xi^{(j)}(y^{(j)}(x))&=x, \;x\in\bar\om_1,\\
y^{(j)}(\xi^{(j)}(z))&=z,\;z\in\bar\om_2.
\end{align*}
By (H2) and the representation \eqref{inverseform} 
\begin{align}\sup_j\int_{\om_1}h(\det Dy^{(j)}(x))\,dx&<\infty,\\\sup_j\int_{\om_2}H(\det(D\xi^{(j)}(z))\,{\mathrm d}z&<\infty,\label{Hest}
\end{align} 
with $H(\delta)=h(\delta^{-1})\delta$, and $\lim_{\delta\to\infty}\frac{H(\delta)}{\delta}=\infty$. 
Hence by Theorem \ref{W1n} applied to $y^{(j)}$ and $\xi^{(j)}$ we have that $y^{(j)}\to y$ uniformly in $\om_1$ and $\xi^{(j)}\to\xi$ uniformly in $\om_2$, so that we can pass to the limit to obtain
\begin{align*}\xi(y(x))=x,\; &x\in\bar\om_1,\\y(\xi(z))=z,\;&z\in\bar\om_2.
\end{align*}
Hence $y:\om_1\to\om_2$ is a homeomorphism with inverse $\xi$, and therefore $y\in\mathcal A_n$.

In order to show that $y$ is a minimiser it therefore suffices to prove the lower semicontinuity property
\be
\label{lsc}
I_{P_1,P_2}(y)\leq \liminf_{j\to\infty}I_{P_1,P_2}(y^{(j)}).
\ee

\begin{lemma}
\label{c2conv}
For some subsequence (not relabelled) 
$c_2(y^{(j)}(x))\to c_2(y(x))$ a.e. in $\om_1$.
\end{lemma}
\begin{proof}
By (H4) we can extract a subsequence with $c_2(y^{(j)}(\cdot))\weakstar c(\cdot)$ in $L^\infty(\om_1,\R^m)$. Let $\varphi\in C_0^\infty(\om_1)$. (Here and below, for $\om\subset\R^n$ open, $C_0^\infty(\om)$ denotes the set of $C^\infty$ functions $\varphi:\om\to\R$ with compact support in $\om$.) Then by \eqref{change}
\begin{align*}
\int_{\om_1}c_2(y^{(j)}(x))\varphi(x)\,dx&=\int_{\om_2}c_2(z)\varphi(\xi^{(j)}(z))\det D\xi^{(j)}(z)\,{\mathrm d}z\\
&=\int_{\om_2}c_2(z)(\varphi(\xi^{(j)}(z))-\varphi(\xi(z)))\det D\xi^{(j)}(z)\,{\mathrm d}z\\
&\hspace{2in}+\int_{\om_2}c_2(z)\varphi(\xi(z))\det D\xi^{(j)}(z)\,{\mathrm d}z.
\end{align*}
The first integral tends to zero since $\varphi(\xi^{(j)}(z))-\varphi(\xi(z))\to 0$ uniformly and $\det D\xi^{(j)}$ is bounded in $L^1(\om_2)$.  The second integral tends to $\displaystyle\int_{\om_2} c_2(z)\varphi(\xi(z))\det D\xi(z)\,dz$ since by \eqref{Hest} $\det D\xi^{(j)}$ is weakly relatively compact in $L^1(\om_2)$ so that by the weak continuity of minors  (see \cite{j8,j17,reshetnyak67a,reshetnyak68}) $\det D\xi^{(j)}\weak \det D\xi$ in $L^1(\om_2)$.
Hence
	\begin{align*}
		\int_{\om_1}c_2(y^{(j)}(x))\varphi(x)\,\dx&\to\int_{\om_2} c_2(z)\varphi(\xi(z))\det D\xi(z)\,{\mathrm d}z=\int_{\om_1}c_2(y(x))\varphi(x)\,\dx.
	\end{align*}
Thus $c_2(y^{(j)})\to c_2(y)$ in the sense of distributions, and so $c_2(y)=c$.
Now apply the same argument to $|c_2(y^{(j)}(x))|^2$, to deduce that $|c_2(y^{(j)}(x))|^2\to |c_2(y)|^2$ in the sense of distributions, and hence in $L^\infty (\om_1)$ weak* (for a further subsequence). Hence $c_2(y^{(j)})\to c_2(y)$ strongly  in $L^2(\om_1,\R^m)$, and thus (for a still further subsequence) a.e. in $\om_1$.
\end{proof}
	\begin{lemma}[\rm see \cite{eisen}]
\label{limsuplem}
Let $E_j\subset \om_1$ be a sequence of measurable sets satisfying\\
$\mL^n( E_j)\geq \delta>0\text{ for all }j.$
Then
$$\mL^n\left(\limsup_{j\to\infty}E_j\right)\geq\delta,$$
where $\displaystyle\limsup_{j\to\infty}E_j:= \bigcap_{k=1}^\infty\bigcup_{j=k}^\infty E_j$.
\end{lemma}
\begin{proof}
Let $E^{(k)}=\displaystyle \bigcup_{j=k}^\infty E_j$. Then
$\om_1\supset E^{(1)}\supset E^{(2)}\supset \cdots$
and $\mL^n(E^{(k)})\geq\delta$, $\mL^n(\om_1)<\infty$. Hence
$$\mL^n\left(\limsup_{j\to\infty}E_j\right)=\mL^n\left(\bigcap_{k=1}^\infty E^{(k)}\right)\geq \delta,$$
as required.
\end{proof}
\begin{lemma}
	\label{lemma3}
	Define $h_j(x)=\psi(c_1(x),c_2(y^{(j)}(x)),Dy^{(j)}(x))-\psi(c_1(x),c_2(y(x)),Dy^{(j)}(x))$. Then $h_j\to 0$ in measure in $\om_1$.
\end{lemma}
\begin{proof}
	If not, there exist $\ep>0, \delta>0$ such that (for a further subsequence) the set
	$$S_j:=\{x\in\om_1:|h_j(x)|>\ep\}$$
	is such that $\mL^n(S_j)>\delta$ for all $j$. We also have that 
	$$\int_{\om_1}\psi(c_1(x),c_2(y^{(j)}(x)), Dy^{(j)}(x))\,\dx\leq a<\infty,\;\;\int_{\om_1}|Dy^{(j)}|\,\dx\leq b<\infty$$
for constants $a,b$.
	Hence there exists $M>0$ such that for all $j$
	$$\mL^n(\{x:\psi(c_1(x),c_2(y^{(j)}(x)),Dy^{(j)}(x))+|Dy^{(j)}(x)|>M\})<\frac{\delta}{2}.$$
	Define
		\begin{align*}
			E_j&=\{x\in\om_1:|h_j(x)|>\ep,
			\;\psi(c_1(x),c_2(y^{(j)}(x)),Dy^{(j)}(x))+|Dy^{(j)}(x)|\leq M\}
		\end{align*}
	Then $\mL^n(E_j)\geq \delta-\frac{\delta}{2}=\frac{\delta}{2}>0$. By Lemma \ref{limsuplem}, $\displaystyle\mL^n\left(\limsup_{j\to\infty} E_j\right)\geq\frac{\delta}{2}$. Let $x\in \displaystyle\limsup_{j\to\infty}E_j$ and be such that  (see Lemma \ref{c2conv}) $c_2(y^{(j)}(x))\to c_2(y(x))$. Then there exist a subsequence $j_k\to\infty$ and  $A\in M^{n\times n}$ with
	$$Dy^{(j_k)}(x)\to A, \;\psi(c_1(x), c_2(y^{(j_k)}(x)), Dy^{(j_k)}(x))\to l<\infty$$
	but
		\be
		\label{now}
			|h_{j_k}(x)|=|\psi(c_1(x),c_2(y^{(j_k)}(x)),Dy^{(j_k)}(x))-\psi(c_1(x),c_2(y(x)),Dy^{(j_k)}(x))|\geq\ep.
		\ee
By (H2) $A\in M^{n\times n}_+$ and so 	$l=\psi(c_1(x), c_2(y(x)),A)$. Passing to the limit in \eqref{now} we get a contradiction. 
\end{proof}
The remainder of the proof follows the same pattern as in \cite[Theorem 4.1]{j17}, \cite[Theorem 6.1]{j26}, but we give the details for the convenience of the reader.
We first claim  that $\det(Dy(x)) > 0$ almost everywhere. Indeed, we have $\det Dy^{(j)}(x) \geq 0$ and $\det Dy^{(j)} \rightharpoonup \det Dy\geq 0$ in $L^1(\Omega_1)$. Let 
		\be
			E = \{ x \in \Omega_1 \, : \, \det Dy(x) = 0\}.
		\ee
Then
		\be
			\lim_{j\to\infty}\int_{E}\det Dy^{(j)}(x) \dx =\int_E\det Dy(x)\,\dx=0,
		\ee
	and so for a subsequence (not relabelled) $\det Dy^{(j)}(x) \to 0$ almost everywhere in $E$, and hence by (H2) $\lim_{j\to\infty}\psi(c_1(x),c_2(y^{(j)}(x)), Dy^{(j)}(x))=\infty$. By Fatou's lemma and (H2) we thus have
		\be
			\liminf_{j \to \infty}\int_E \psi(c_1(x),c_2(y^{(j)}(x)), Dy^{(j)}(x))\,\dx &&\\
			&&\hspace{-1in}\geq 
			\int_E \liminf_{j \to \infty}\psi(c_1(x),c_2(y^{(j)}(x)), Dy^{(j)}(x))\,\dx
			.\nonumber
		\ee
	Hence $\mL^n(E)=0$, proving the claim.
	
	To complete the proof, we denote by
		\[
			v^{(j)} = ({\mathbf J}(Dy^{(j)}),\det Dy^{(j)})
		\]
	the vector of minors of $Dy^{(j)}$, which by the weak continuity of minors converges weakly to 
	$v = ({\mathbf J}(Dy),\det Dy)$  in $L^1(\om_1,\R^{\sigma(n)+1})$. By Mazur's theorem 
 (see e.g. \cite[Corollary 3.8]{brezis2011}), there exists a sequence
		\begin{equation}
			w^{(k)} = \sum_{j = k}^{N_k}\lambda_j^{(k)}v^{(j)}
		\end{equation}
	of convex combinations of the $v^{(j)}$ converging strongly to $v$ in $L^1(\Omega_1,\R^{\sigma(n)+1}))$  and hence (extracting a further subsequence) almost everywhere in $\Omega_1$. 
	By the convexity of $g$ and (H3), we have that for a.e. $x\in\om_1$
		\begin{align}\nonumber
			g(c_1(x),c_2(y(x)),w^{(k)}(x)) 
			+ 
			\sum_{j = k}^{N_k}\lambda_j^{(k)}h_j(x)
			&\leq
			\sum_{j = k}^{N_k}\lambda_j^{(k)}\left(h_j(x) 	+g(c_1(x),c_2(y(x)),v^{(j)}(x))\right)
			\\ \label{mazurcombo}
			&=
			\sum_{j = k}^{N_k}\lambda_j^{(k)}\left(h_j(x) 	+\psi(c_1(x),c_2(y(x)),Dy^{(j)}(x))\right).
		\end{align}
By Lemma \ref{lemma3} we may suppose that $h_j(x)\to 0$ a.e. in $\om_1$.	Hence taking the $\liminf_{k\to \infty}$ of both sides in \eqref{mazurcombo} we obtain that for a.e. $x\in\om_1$
		\begin{align}\nonumber
			\psi(c_1(x),c_2(y(x)),Dy(x)) 
			&=
			g(c_1(x),c_2(y(x)),{\mathbf J}(Dy(x)),\det Dy(x)) 
			\\ \label{ineqa}
			&\leq
			\liminf_{k\to\infty} \sum_{j = k}^{N_k} \lambda_j^{(k)} 	\psi(c_1(x),c_2(y^{(j)}(x)),Dy^{(j)}(x)) 
			\\ \nonumber
			&\leq
			\liminf_{j\to\infty} \psi(c_1(x),c_2(y^{(j)}(x)),Dy^{(j)}(x)).
		\end{align}
	Integrating \eqref{ineqa} over $\Omega_1$ and applying Fatou's lemma we obtain  \eqref{lsc} as required.
\end{proof}
\begin{remark}
\label{nonempty}\rm 
A sufficient condition that $\inf_{y\in\mA_p}I_{P_1,P_2}(y)<\infty$ is that $\mA_\infty$ is nonempty. Indeed, if $w\in\mA_\infty$ with inverse $\zeta$ then, since $y$ maps sets of measure zero to sets of measure zero, $\det Dw(x)  \det D\zeta(w(x))=1$ for a.e. $x\in\om_1$. Since $D\zeta$  is bounded in $L^\infty$ we have $\det Dy(x)\geq \mu>0$ a.e.  for some $\mu$, and so $\psi(c_1(x),c_2(w(x)),Dw(x))\leq M<\infty$ for some $M$.
\end{remark}
\subsection{Landmark based registration}\label{landmark}
We now show that for $n\geq 2$ a minimiser exists under the constraint that $y$ maps a finite number $N\geq 1$ of distinct landmark points $p_1,\ldots, p_N$ belonging to $\om_1$ to corresponding distinct landmark points $q_1,\ldots,q_N$ in $\om_2$. We define the corresponding set of admissible maps
$$\bar\mA_n:=\{y\in\mA_n: y(p_i)=q_i \text{ for }i=1,\ldots,N\}.$$
			\begin{theorem}\label{thm:landmark}
Suppose that $n\geq 2$, that {\rm  (H1)-(H4)} hold and that $\mA_\infty$ is nonempty. Then there exists an absolute minimiser $y$ of $I_{P_1,P_2}$ in $\bar\mA_n$, and $I_{P_1,P_2}(y)<\infty$.
			\end{theorem}
			\begin{proof}
By assumption there exists some $w\in\mA_\infty$. Let $\tilde q_i=w^{-1}(q_i), i-1,\ldots,N$. Since $n\geq 2$, by   \cite[Lemma 2.1.10]{banyaga}, \cite{michor94},\cite[p472]{KrieglMichor} the group of smooth diffeomorphisms of $\om_1$ which are equal to the identity outside some compact set is $N$-transitive. Thus there exists a smooth diffeomorphism $\varphi:\om_1\to\om_1$ which is equal to the identity outside some compact subset of $\om_1$ and such that $\varphi(p_i)=\tilde q_i$ for each $i$. In particular $D\varphi$ and $D(\varphi^{-1})$ are uniformly bounded in $\om_1$. Define $\tilde y=w^{-1}\circ \varphi$. Then $\tilde y\in W^{1,\infty}(\om_1,\R^n), \tilde y^{-1}\in W^{1,\infty}(\om_2,\R^n)$ and $\tilde y(p_i)=q_i$ for each $i$. Hence $\bar\mA_n$ is nonempty and $\inf_{\bar\mA_n}I_{P_1,P_2}<\infty$. The result now follows because for a minimising sequence $y^{(j)}$ of $I_{P_1,P_2}$ in $\bar\mA_n$ as in Theorem \ref{exthm} with $y^{(j)}\weak y$ in $W^{1,n}(\om_1,\R^n)$ we have $y^{(j)}(p_i)=q_i\to y(p_i)$ for each $i$, and hence $y\in\bar\mA_n$.		\end{proof}
\begin{remark}\rm  If we assume only that $p_i\in\bar\om_1$,  $q_i\in\bar\om_2$ then $\tilde\mA_n$ may be empty. Suppose that $p_i\in\partial\om_1$ for $1\leq i\leq k$, $p_i\in\om_1$ otherwise. Then  by Theorem \ref{Apequiv} a necessary condition for $\tilde\mA_n$ to be nonempty (so that a minimiser exists) is that 
\be
\label{nec}
q_i\in\partial\om_2\text{ for }1\leq i\leq k,\; q_i\in\om_2\text{ otherwise}.
\ee

 Consider the case when $\om_1=B_n:=\{x\in\R^n:|x|<1\}$ (equivalently for any image $\om_1$ of $B_n$ under an orientation-preserving bi-Lipschitz map, such as an $n-$cube). Then if $n=2$ the necessary condition \eqref{nec} is not sufficient. For example let $\om_1=\om_2=B_2$ and $N=k=4$ with $p_1=q_1=(0,1)$, $p_2=q_3=(0,1)$, $p_3=q_2=(-1,0)$, $p_4=q_4=(0,-1)$. Then any  $y\in \tilde\mA_2$  restricted to $S^1=\partial B_2$ would be a homeomorphism of $S^1$ with $y(p_i)=q_i, 1\leq i\leq 4$, which is impossible since then the connected image under $y$ of the arc $C:=\{(\cos\theta,\sin\theta):0\leq\theta\leq \frac{\pi}{2}\}$  would contain either $q_3$ or $q_4$, so that either $p_3\in C$ or $p_4\in C$, which is not the case.

However if $n\geq 3$  then \eqref{nec} {\it is} sufficient. Following the proof of Theorem \ref{thm:landmark} it is enough to prove this in the case $\om_1=\om_2=B_n$. Applying the result on transitivity cited in the proof to the manifold $S^{n-1}=\partial B_n$, whose dimension is $\geq 2$, we obtain a smooth diffeomorphism $f$ of $S^{n-1}$ with $f(p_i)=q_i$ for $1\leq i\leq k$. The radial extension 
$$\Phi(x)=|x|f\left(\frac{x}{|x|}\right)$$
is a bi-Lipschitz self-map of the closed ball $\bar B_n$. If $k=N$ then $\Phi\in\tilde\mA_n$. If $k<N$ then, applying $(N-k)-$transitivity for diffeomorphisms of $B_n$, there exists  a diffeomorphism $\tilde \varphi$ of $B_n$ equal to the identity near $S^{n-1}$ such that $\tilde\varphi(\Phi(p_i))=q_i$ for $i=k+1,\ldots,N$, so that $\tilde\varphi\circ\Phi \in \tilde\mA_n$.  (We remark that  $f$ cannot in general be extended to a diffeomorphism of $\bar B_n$, though, for example, a theorem of Smale \cite{smale59} (see Thurston \cite[Theorem 3.10.11]{thurston}) states that this is possible for $n=3$.)
\end{remark}
\subsection{Comparing one image with part of another}
\label{part}
	In this subsection we regard $P_1 = (\Omega_1,c_1)$ as a template image to be detected as a homeomorphic part of the image $P_2 = (\Omega_2,c_2)$, allowing for changes of scale, perspective, or other affine transformations. (Because in this situation $P_1$ and $P_2$ haves different roles, integrands $\psi$ that do not satisfy (C3) may be appropriate.)
	
	Let $S$ be a relatively closed subset of $M^{n \times n}_+$. We consider the minimisation of $I_{P_1,P_2}(y)$ subject to the constraint that $y(\Omega_1) =a+ A\Omega_1  \subset \Omega_2$, where $a \in \R^n$ and $A \in S$, and establish existence of minimisers. 
	
	Interesting examples with relevance to particular object detection and recognition problems include
		\begin{enumerate}[label={(\roman*)}]\label{S:examples}
			\item {$S = \{\mu \1 : 0 < \lambda \leq \mu \leq \Lambda < \infty\}$ for fixed $\lambda,\Lambda$, detecting an undistorted template image within a target image at a range of scalings (or a fixed scaling if $\lambda= \Lambda$);}
			\item{ $S = \{\mu R : R \in SO(n), 0 < \lambda \leq \mu \leq \Lambda < \infty\}$ for fixed $\lambda,\Lambda$, allowing for rotation, scaling and magnification/reduction within a certain closed range.}
			\item{$S = M^{n\times n}_+$, allowing for any affine transformation of $\om_1$.}
		\end{enumerate}	
We define the set of admissible maps by
\begin{align*}
			\mathcal{A}_{n,S} \eqdef 
			\{
			y \in W^{1,n}(\Omega_1,\R^n): y:\om_1\to a+A\om_1\subset\om_2  \text{ a homeomorphism}&\\ &\hspace{-2.3in}\text{for some }a\in\R^n, A\in S \text{ with }y^{-1}\in W^{1,n}(a+A\om_1,\R^n)\}.
			\end{align*}
		\begin{theorem}\label{thm:closedexistence}
			Suppose  {\rm (H1)-(H4)} hold and that $\mathcal{A}_{n,S}$ is nonempty. Then there exists an absolute minimiser of $I_{P_1,P_2}$ in  $\mathcal{A}_{n,S}$. 
		\end{theorem}
		\begin{proof}
			As in Theorem \ref{exthm} we consider a minimising sequence $y^{(j)}$ for which 
			\begin{equation}
				\lim_{j\to\infty}I_{P_1,P_2}(y^{(j)})=  \inf_{y \in \mathcal{A}_{n,S} } I_{P_1,P_2}(y):=l<\infty.
			\end{equation}
	(If desired we can take $y^{(j)}$ to be for each $j$ a minimiser of $I_{P_1,P_2}(y)$ in $\mA(\om_1, a^{(j)}+A^{(j)}\om_1)$, whose existence is guaranteed by Theorem \ref{exthm}, but this does not simplify the proof.) We have that $y^{(j)}(\Omega_1) =  a^{(j)}+A^{(j)}\Omega_1\subset \om_2$ for some $a^{(j)}\in\R^n, A^{(j)}\in M^{n\times n}_+$. We first claim that $a^{(j)}$ and $A^{(j)}$ are bounded. Let $B(x_0,\ep)\subset\om_1$ for some $x_0$ and $\ep>0$. Since $\om_2$ is bounded we have that $|a^{(j)}+A^{(j)}(x_0+\ep z)|\leq M<\infty$ whenever $|z|<1$. Hence $|A^{(j)}\ep z|\leq|a^{(j)}+A^{(j)}(x_0+\ep z)-(a^{(j)}+A^{(j)}x_0)|\leq 2M$, from which it follows that $|A^{(j)}|\leq 2M\ep^{-1}$, so that  the  $A^{(j)}$, and hence also the $a^{(j)}$, are bounded. Therefore we may assume that $A^{(j)}\to A$ and $a^{(j)}\to a$ as $j\to \infty$. 

Since
			\begin{equation}
				\int_{\Omega_1} \det Dy^{(j)} \,\dx =\mL^n(a^{(j)}+A^{(j)}\om_1) =(\det A^{(j)})\mL^n(\Omega_1),
			\end{equation}
			 if $\det{A^{(j)}} \to 0$, then $\det Dy^{(j)} \to 0$ in $L^1(\Omega_1)$ and so for a subsequence, not relabelled, $\det Dy^{(j)} \to 0$ a.e. in $\Omega_1$. Hence, by Fatou's Lemma,
			$$\infty=\int_{\Omega_1} \liminf_{j \to \infty} \,h(\det Dy^{(j)})\,dx 	\leq 	
				\liminf_{j \to \infty}I_{P_1,P_2}(y^{(j)})<\infty.$$ 
 This contradiction proves that $\det A>0$, so that $A\in S$.		

Now define $B^{(j)}:=A^{(j)}A^{-1}$ and $\tilde y^{(j)}:\om_1\to\R^n$	by
\be
\tilde y^{(j)}(x):=a+(B^{(j)})^{-1}(y^{(j)}(x)-a^{(j)}),
\ee
so that $Dy^{(j)}(x)=B^{(j)}D\tilde y^{(j)}(x)$ and  $\tilde y^{(j)}(\om_1)=a+A\om_1\subset\om_2$. It is easily checked that $\tilde y^{(j)}:\om_1\to a+A\om_1$ is a homeomorphism for each $j$. Note that $\lim_{j\to\infty}B^{(j)}=\1$, where $\1$ denotes the identity matrix in $M^{n\times n}$. Then
\be
I_{P_1,P_2}(y^{(j)})=\int_{\om_1}\psi(c_1(x),c_2(a^{(j)}+B^{(j)}(\tilde y^{(j)}(x)-a)), B^{(j)}D\tilde y^{(j)}(x))\,\dx.
\ee

As in Theorem \ref{exthm} we may suppose that 
$$\tilde y^{(j)}\weak  y \text { in } W^{1,n}(\om_1,\R^n),\;
\tilde \xi^{(j)}\weak \xi\text { in } W^{1,n}(a+A\om_1,\R^n),$$
where $\tilde \xi^{(j)}:=(\tilde y^{(j)})^{-1}$,
and it follows as before that $ y:\om_1\to a+A\om_1$ is a homeomorphism with inverse $\xi$. A  modification of Lemma \ref{c2conv} shows that for a further subsequence
$c_2(a^{(j)}+B^{(j)}(\tilde y^{(j)}(x)-a))\to c_2(y(x))$ a.e. in $\om_1$. The key point is to show that for $\varphi\in C_0^\infty(\om_1)$
\be\nonumber
\lim_{j\to\infty}\int_{a+A\om_1}c_2(a^{(j)}+B^{(j)}(z-a))\varphi(\xi(z))\det D\tilde\xi^{(j)}(z)\,{\mathrm d}z=\int_{a+A\om_1}c_2(z)\varphi(\xi(z))\det D\xi(z)\,{\mathrm d}z.
\ee
But this follows since $\det D\tilde\xi^{(j)}\weak \det D\xi$ in $L^1(\a+A\om_1)$ and (extracting a further subsequence) 
$c_2(a^{(j)}+B^{(j)}(\cdot-a))\to c_2$ boundedly a.e. in $\a+A\om_1$.
	The remainder of the proof of Theorem \ref{exthm} is then easily adapted to show that $y$ is a minimiser.	\end{proof}	

\subsection{Image morphing and metrics}
\label{metricsection}
Motivated by the work of Berkels, Effland and Rumpf \cite{BerkelsEfflandRumpf2015} we indicate in this section some connections of our theory to the morphing of images and corresponding metrics between images.

We assume that $\om=\om_1=\om_2$ is fixed, so that  images are described by  maps $c \in L^\infty(\om,K)$. Suppose that $\psi$ satisfies (H1)-(H3) together with 
$$\psi(c_1,c_2,A)=\psi(c_2,c_1,A^{-1})\det A,$$
and $\psi(c_1,c_2,A)=0$ iff $c_1=c_2$ and $A\in SO(n)$. 
 Let $P:[0,\infty)\to [0,\infty)$ be a strictly increasing continuous function with $P(0)=0$, and define
$$F(c,d):=P\left(\min_{y\in\mA_n}\int_\om\psi(c(x),d(y(x)),Dy(x))\,\dx\right).$$
Then $F(c,d)$ is finite (since the identity map belongs to $\mA_n$), and satisfies $0\leq F(c,d)=F(d,c)$. Furthermore $F(c,d)=0$ iff for some $y\in \mA_n$,  $Dy(x)\in SO(n)$   and $c(x)=d(y(x))$ for a.e. $x\in\om$. But (see    \cite{ballhorner}) this implies by \cite{reshetnyak67a} that $Dy(x)=R$ a.e. for some (constant) $R\in SO(n)$, so that $y(x)=R x+a$ for some $a\in\R^n$, and in particular $\om=R\om+a$. Suppose first that $\om$ has no nontrivial Euclidean symmetry, so that $\om=R\om+a$ implies $R=\1$ and 
hence $a=0$. Thus $F(c,d)=0$ iff $c=d$. 

Given an integer $N\geq 1$ and $c,d\in L^\infty(\om,K)$ let
$${\mathcal C}_N(c,d):=\{\cv=(c_1,c_2,\ldots,c_{N+1}): \text{each }c_i\in L^\infty(\om,K), c_1=c, c_{N+1}=d\},$$
and define
$$F_N(c,d)=\inf_{{\cv\in\mathcal C}_N(c,d)}\sum_{i=1}^{N}F(c_{i},c_{i+1}).$$
Then $F_N(c,d)=F_N(d,c)\geq 0$ and $F_N(c,c)=0$. Also, since we can take $c_{N+2}=c_{N+1}=d$ we have that 
$$F_{N+1}(c,d)\leq F_N(c,d)\leq F_1(c,d)=F(c,d).$$
\begin{lemma}
\label{pseudometric}
For $N\geq 1, M\geq 1$ and $c,d,e\in L^\infty(\om,K)$ we have that 
\be
\label{psm}F_{N+M}(c,e)\leq F_N(c,d)+F_M(d,e).
\ee
\end{lemma}
\begin{proof}
Given $\ep>0$  there exist $c_i\in L^\infty(\om,K)$  with $c_1=c, c_{N+1}=d, c_{N+M+1}=e$ and
$$\sum_{i=1}^{N}F(c_i,c_{i+1})<F_N(c,d)+\ep,\;\sum_{i=N+1}^{M+N}F(c_i,c_{i+1})<F_M(d,e)+\ep.$$
Hence $F_{N+M}(c,e)\leq  F_N(c,d)+F_M(d,e)+2\ep$ and the result follows by  letting $\ep\to 0+$.
\end{proof}
Now define for $c,d\in L^\infty(\om,K)$
\be
\label{metric}
\rho(c,d):=\inf_{N\geq 1}F_N(c,d)=\lim_{N\to\infty}F_N(c,d).
\ee
Letting $N,M\to\infty$ in \eqref{psm} we deduce from Lemma \ref{pseudometric} that $\rho$ is a pseudometric on $L^\infty(\om,K)$, that is $\rho(c,d)=\rho(d,c)\geq 0$, $\rho(c,c)=0$ and $\rho$ satisfies the triangle inequality
$$\rho(c,e)\leq \rho(c,d)+\rho(d,e)\text{ for all }c,d,e\in L^\infty(\om,K).$$
However (as pointed out by a referee) $\rho$ is not necessarily a metric, for which we need also that $\rho(c,d)=0$ implies $c=d$. For example, in the case when $\psi$ is given by \eqref{8aa}, \eqref{8a} with $g(c_1,c_2)=|c_1-c_2|^2$ we have that $\rho(c,d)=0$ for {\it any} $c,d$. Indeed we can choose 
$$c_i=\left(1-\frac{i-1}{N}\right)c+\frac{i-1}{N}d, \;1\leq i\leq N+1,$$
when, for $y(x)=x$ we have that
$$\int_\om(\Psi(Dy(x))+(1+\det Dy(x))|c_i(x)-c_{i+1}(y(x))|^2)\,dx=\frac{2}{N^2}\int_\om|c-d|^2\,dx$$
and hence, taking $P(t)=t$,
$$F_N(c,d)\leq \frac{2}{N}\int_\om|c-d|^2dx\to 0 \text{ as }N\to\infty.$$ In this case, or when $\psi$ is given by \eqref{8aa}, \eqref{8b}, it seems plausible that, under suitable hypotheses on $\Psi$,  $\rho$ is a metric if we take $P(t)=t^\half$, but this seems to be a delicate question involving compositions of Sobolev homeomorphisms.

A related issue is whether the infimum in the definition of $F_N(c,d)$ is attained. To this end we impose a strengthened polyconvexity condition\vspace{.05in}\\
 (H3)$'$ 
 there is a convex function $g:\R^m\times \R^m\times \R^{\sigma(n)}\times (0,\infty)\to \R$  such that
$$\psi(c,d,A)=g(c,d\det A,{\mathbf J}_{n-1}(A), \det A)\text{ for all }c,d\in K, A\in M^{n\times n}_+.$$
 Note that by Remark \ref{convexity}  (H3)$'$ is satisfied in the mass-preserving case \eqref{8aa}, \eqref{8b} provided $\Psi$ is polyconvex. However it is not hard to check that (H3)$'$  is not satisfied in the case \eqref{8aa}, \eqref{8a}.

\begin{lemma}
\label{lsc1}
Suppose  {\rm (H1),(H2),}{\rm (H3)}$'$ hold and that $c^{(j)}\weakstar c,\,d^{(j)}\weakstar d$ in $L^\infty(\om,\R^m)$. Then
\be
\label{lsc2}
F(c,d)\leq\liminf_{j\to\infty}F(c^{(j)},d^{(j)})
\ee
\end{lemma}
\begin{proof}
By Theorem \ref{exthm}, for each $j$ there exists a minimiser $y^{(j)}\in \mA_n$ of
$$I^{(j)}(y):=\int_\om\psi(c^{(j)}(x),d^{(j)}(y(x)),Dy(x))\,\dx,$$
so that $F(c^{(j)},d^{(j)})=P(I^{(j)}(y^{(j)}))$,  By the arguments in the proof of Theorem \ref{exthm} we may assume that $y^{(j)}\weak y$ and $\xi^{(j)}\weak \xi$  in $W^{1,n}(\om,\R^n)$, where  $\xi^{(j)}:=(y^{(j)})^{-1}$ and $\xi=y^{-1}$, and that both sequences converge uniformly. We note that, following the argument in the proof of Lemma \ref{c2conv},  $d^{(j)}(y^{(j)})\det Dy^{(j)}\weak d(y)\det Dy$ in $L^1(\om,\R^m)$. Indeed 
if $\varphi\in C_0^\infty(\om)$ then
\begin{align*}
\lim_{j\to\infty}\int_\om d^{(j)}(y^{(j)}(x))\det Dy^{(j)}(x)\varphi(x)\,dx&=\lim_{j\to\infty}\int_\om d^{(j)}(z)\varphi(\xi^{(j)}(z))\,dz\\
&=\int_\om d(z)\varphi(\xi(z))\,dz\\
&=\int_\om d(y(x))\det Dy(x)\varphi(x)\,dx,
\end{align*}
 so that $d^{j)}(y^{(j)})\det Dy^{(j)}\to d(y)\det Dy$ in the sense of distributions, and hence by (H2) weakly in $L^1(\om,\R^m)$. Therefore, by (H3)$'$ and the arguments in the proof of Theorem \ref{exthm}
\begin{align*}
\liminf_{j\to\infty}F(c^{(j)},d^{(j)})&=P\left(\liminf_{j\to\infty}\int_\om g(c^{(j)},d^{(j)}(y^{(j)})\det Dy^{(j)}, {\mathbf J}(Dy^{(j)}),\det Dy^{(j)})\,dx\right)\\
&\geq P\left(\int_\om g(c,d\det Dy,{\mathbf J}(Dy),\det Dy)\,dx\right)\\
&\geq F(c,d),
\end{align*}
as claimed.
\end{proof}
As an immediate consequence we can prove the existence of a `discrete  morphing sequence' between a pair of images $c,d$.
\begin{theorem}
\label{discrete}Suppose  {\rm (H1),(H2),}{\rm (H3)}$'$ hold. 
Then given $c,d\in L^\infty(\om,K)$ the function
\be
\label{FN}
I_N(c_1,\ldots,c_{N+1}):=\sum_{i=1}^{N}F(c_i,c_{i+1})
\ee
attains a minimum in ${\mathcal C}_N(c,d)$. In particular $F_N(c,d)=0$ iff $c=d$.
\end{theorem}
\begin{proof}
Let $(c_1^{(j)},\ldots,c_{N+1}^{(j)})$ be a minimizing sequence for $I_N$ in ${\mathcal C}_N(c,d)$. Since $K$ is bounded and convex we may assume that each $c_i^{(j)}\weakstar c_i$ in $L^\infty(\om,K)$, and it then follows from Lemma \ref{lsc1} that $(c_1,\ldots,c_{N+1})$ is a minimiser.

If $F_N(c,d)=0$ then for a minimiser ${\mathbf c}=(c_1,\ldots,c_{N+1})$ we have $F(c_i, c_{i+1})=0$ for each $i$ and hence $c=c_1=c_2=\cdots c_{N+1}=d$.
\end{proof}

Theorem \ref{discrete} should be compared to that of Berkels, Effland and Rumpf \cite[Theorem 3.4]{BerkelsEfflandRumpf2015}, who consider an energy functional of the type
$$I_{c,d}(\psi)=\int_\om (\Psi(Dy) +\gamma |D^my|^2+\frac{1}{\Delta}|c-d(y)|^2)\,dx,$$
where $m>1+\frac{n}{2}$ and $\Delta>0$. The smoothness induced by the term $\gamma |D^my|^2$,
together with the convexity of the integrand in $c,d$, enable them to prove the existence of  minimisers for each $N$  of 
$$I_N(\cv):=\sum_{i=1}^{N}I_{(c_i,c_{i+1})}(y_i)$$
among discrete paths of images $\cv=(c_1,\ldots,c_N)$,  $c_1=c, c_{N+1}=d$ and corresponding diffeomorphisms ${\mathbf y}=(y_1,\ldots,y_{N})$
  with $y_i(x)=x$ for $x\in\partial\om$.  Theorem \ref{discrete} shows that the higher-order derivative term is not needed in the mass-conserving case, and that we can also allow sliding at the boundary.

Berkels, Effland and Rumpf \cite{BerkelsEfflandRumpf2015} (see also  \cite{NeumayerPerschSteidl2018}, \cite{EfflandNeumayerRumpf2020}, \cite{guastinietal23}) then show how to obtain a metamorphosis model of the type considered by Miller and Younes \cite{MillerYounes} and by Trouv\'e and Younes \cite{TrouveYounes}, \cite{TrouveYounesa} by passing to the limit $N\to\infty$ in the scaled functional $NI_N$. Without the  higher-order gradient term it does not seem obvious  how to pass to this limit, when one would hope to obtain a fluid model with slip boundary conditions.

If $\om$ has a nontrivial Euclidean symmetry (for example, if $\om$ is a ball or  parallelepiped) then we identify equivalent images by the equivalence relation $c\sim d$ iff $c(x)=d(Rx+a)$ for a.e. $x\in\om$ for some $R\in SO(n)$ and $a\in\R^n$ with  $\om=R\om+a$, and let $X$ be the set of corresponding equivalence classes in $L^\infty(\om,K)$. For two such equivalence classes $[c], [d]$ containing $c,d\in L^\infty(\om,K)$ respectively we can define 
\begin{align}\tilde F([c],[d])&=\inf_{\tilde c\in[c], \tilde d\in[d]}F(\tilde c,\tilde d)\nonumber\\
&=\inf_{\substack{R,\bar R\in SO(n), \;a,\bar a\in\R^n\\R\om+a=\bar R\om+\bar a=\om}}P\left(\min_{y\in\mA_n}\int_\om\psi(c(Rx+a),d(\bar Ry(x)+\bar a),Dy(x))\,dx\right).\label{newf}
\end{align}
Clearly $\tilde F([c],[d])=\tilde F([d],[c])\geq 0$. 
We claim that there exist $R^*,\bar R^*\in SO(n),\,a^*, \bar a^*\in \R^n$ with $R^*\om+a^*=\bar R^*\om+\bar a^*=\om$, and $y^*\in\mA_n$ with 
$$\tilde F([c],[d])=P\left(\int_\om \psi(c(R^*x+a^*),d(\bar R^*y^*(x)+\bar a^*),Dy^*(x))\,dx\right).$$
To this end let $R^{(j)}, \bar R^{(j)},a^{(j)},\bar a^{(j)}, y^{(j)}$ be a minimizing sequence for \eqref{newf}, where we may assume that $R^{(j)}\to R^*, \bar R^{(j)}\to \bar R^*, a^{(j)}\to a^*, \bar a^{(j)}\to \bar a^*$, so that
$R^*\om+a^*=\bar R^*\om+\bar a^*=\om$. Furthermore we may suppose that $y^{(j)}\weak y^*$ in $W^{1,n}(\om,\R^n)$ for some $y^*\in\mA$, and we need to show that 
\begin{align*}\int_\om\psi(c(R^*x+a^*),d(\bar R^*y^*(x)+\bar a^*),Dy^*(x))\,dx&\\\ &\hspace{-1.5in}\leq \liminf_{j\to\infty}\int_\om\psi(c(R^{(j)}x+a^{(j)}),d(\bar R^{(j)}y^{(j)}(x)+\bar a^{(j)}),Dy^{(j)}(x))\,dx.
\end{align*}
Following the proof of Theorem \ref{exthm} it suffices to check that (for a subsequence) $c(R^{(j)}x+a^{(j)})\to c(R^*x+a^*)$ and $d(\bar R^{(j)}y^{(j)}(x)+\bar a^{(j)})\to d(\bar R^*y^*(x)+\bar a^*)$ a.e. in $\om$. The first assertion is easily proved (e.g. using the method in Lemma \ref{c2conv} of proving weak and norm convergence in $L^2$), while the second follows from Lemma \ref{c2conv} applied to the map $\tilde y^{(j)}(x)=\bar R^{(j)}y^{(j)}(x)+\bar a^{(j)}$.

Hence if $F([c],[d])=0$ we have that $y^*(x)=\hat R x+\hat a$ for some $\hat R\in SO(n)$, $\hat a\in\R^n$ with $\hat R\om+\hat a=\om$, and thus $c(R^*x+a^*)=d(\bar R^*(\hat Rx+\hat a)+\bar a^*)$, from which it follows that $[c]=[d]$. Then we can define
$$\tilde{\mathcal C}_N([c],[d]):=\{\tilde\cv=([c_1],[c_2],\ldots,[c_{N+1}]):\text{ each }[c_i]\in X, [c_1]=[c], [c_{N+1}]=[d]\}$$
and
\be
\label{tildeFN}
\tilde F_N([c],[d])=\inf_{\tilde\cv\in\tilde{\mathcal C}_N([c],[d])}\sum_{i=1}^N\tilde F([c_i],[c_{i+1}]).
\ee
and applying Lemma \ref{metric} to $\tilde F_N$ we obtain a corresponding pseudometric on $X$. Furthermore, under the hypotheses (H1), (H2), (H3)$'$ we find as in Theorem \ref{FN} that the infimum in \eqref{tildeFN} is attained.

How to handle the case $\om_1\neq \om_2$ is less clear, since it would require considering paths between images for which also the domain interpolates between $\om_1$ and $\om_2$.

\section{Quasiconvexity and affine scaling}
\label{scaling}
\subsection{Affinely related images and a corresponding quasiconvexity condition}
Suppose that the images $P_1=(\Omega_1,c_1)$ and $P_2=(\Omega_2,c_2)$ are affinely related, i.e. for some   $M\in M^{n\times n}_+$ and $a\in\R^n$ we have
\be
\label{lin}\Omega_2=M\Omega_1+a,\;\; c_2(Mx+a)=c_1(x).
\ee
We ask whether  we can choose $\psi$ such that for each such pair of affinely related images the unique minimiser $y$ of $I_{P_1,P_2}$  in $\mA_n(\om_1,\om_2)$ is $y(x)=Mx+a$. Note that without loss of generality we can  assume that $a=0$, since $y$ is a minimiser of $I_{P_1,P_2}$ in $\mA_n(\om_1, M\om_1+a)$ iff $y-a$ is a minimiser of $I_{P_1,\tilde P_2}$ in $\mA_n(\om_1,M\om_1)$, where $\tilde P_2:=(M\om_1,c_2(\cdot+a))$. Therefore we consider below only pairs $P_1=(\om_1,c_1)$ and $P_2=(\om_2,c_2)$ of {\it linearly related} images, i.e.
\be
\label{lina}
\om_2=M\om_1,\;\;c_2(Mx)=c_1(x).
\ee

We consider  $\psi$ of the form \eqref{8aa}, \eqref{8a}, that is
\be
\label{psiform}
\psi(c_1,c_2,A)=\Psi(A)+(1+\det A)g(c_1,c_2),
\ee
where $\Psi:M^{n\times n}_+\to[0,\infty)$ is isotropic with $\Psi(A)=\det A\cdot\Psi(A^{-1})$, $\Psi^{-1}(0)=SO(n)$, and $g:\R^n\times\R^n\to[0,\infty)$ continuous with $g(c_1,c_2)=g(c_2,c_1)$, $g(c_1,c_2)=0$ iff $c_1=c_2$. 
Thus, since $g(c_1(x), c_2(Mx))=0$, we require that  for all   $y\in\mA_n$ with $y(\Omega_1)=M\Omega_1$
\be
\label{Mcond}
\int_{\Omega_1}\left(\Psi(Dy(x))+(1+\det Dy(x))g(c_1(x),c_2(y(x)))\right)\,\dx \geq \displaystyle\int_{\Omega_1}\Psi(M)\,\dx,\;\;
\ee
with equality iff $y(x)=Mx$.

Since $g\geq 0$, a sufficient condition for \eqref{Mcond} to hold is that
\be
\label{qc}\av_{\Omega_1}\Psi(Dy)\,\dx\geq \Psi(M)
\ee
for all $\om_1$ and $y\in\mA_n$, where $\avsmall_{\Omega_1} f\,dx:=\frac{1}{\mL^n(\om_1)}\int_{\om_1}f\,dx$. If we require \eqref{Mcond} to hold for a sufficiently large family of linearly related  image pairs then \eqref{qc} is also necessary. Indeed, given $y\in\mA_n=\mA_n(\om_1,\om_2)$ choose $c:\om_1\to K$ to be a bounded measurable function invariant under the map $M^{-1}y:\om_1\to\om_1$; for example we can take $c$ to be constant, but there can be nonconstant such $c$.  Suppose that \eqref{Mcond} holds for some sequence of image pairs $P_1^{(j)}=(\om_1,c_1^{(j)}), \, P_2^{(j)}=(M\om_1,c_2^{(j)})$, where $c_1^{(j)}:\om_1\to K$,  $c_1^{(j)}(x)\to c(x)$ boundedly a.e. in $\om_1$ and $c^{(j)}_2(z):=c_1^{(j)}(M^{-1}z)$. Then  $c^{(j)}_2(y(x))\to c(M^{-1}y(x))= c(x)$ boundedly a.e. in $\om_1$, so that $g(c^{(j)}_1(x),c^{(j)}_2(x))\to 0$ boundedly a.e. in $\om_1$. Hence passing to the limit $j\to\infty$ in \eqref{Mcond} for $c_1^{(j)},c_2^{(j)}$ we get \eqref{qc} for $y$.

The inequality \eqref{qc} is a stronger version of {\it quasiconvexity at }$M$, in which the usual requirement that $y(x)=Mx$ for $x\in\partial\Omega_1$ is weakened to $y(\om_1)=M\om_1$. The requirement that \eqref{qc} holds for all $M$  and all $y$ with $y(x)=Mx$ for $x\in\partial\Omega_1$ is {\it quasiconvexity}, the central convexity condition of the multi-dimensional calculus of variations, which is implied by polyconvexity (see, e.g. \cite{rindler2018}). (Note, however, that in the calculus of variations literature the quasiconvexity inequality is not usually subject to the extra condition that $y$ is a  homeomorphism. For a recent discussion of this issue see \cite{astalaetal2023}.) 
\subsection{Uniform magnification and rotation}
\label{mag}
We show that we can choose $\Psi$ to satisfy  \eqref{qc} for all $M$ of the form  $M=\lambda R$, $\lambda>0$, $R\in SO(n)$,  so that $P_2$ is a uniform magnification (or reduction if $\lambda\leq 1$) of $P_1$ plus a possible rigid rotation. Let 
\be
\label{isoPsi}
\Psi(A)= \sum_{i=1}^nv_i^\alpha  +
\det A\left(\sum_{i=1}^nv_i^{-\alpha}\right)+h(\det A),
\ee
where $v_i=v_i(A)$ are the singular values of $A$, $\alpha\geq n$, and where $h:(0,\infty)\to\R$ is $C^1$, convex and bounded below with $h(\delta)=\delta h(\delta^{-1})$ and $h'(1)=-n$.
Then $\Psi$ is isotropic, $\Psi(A)=\det A\cdot \Psi(A^{-1})$,
$\Psi\geq 0$, $\Psi^{-1}(0)=SO(n)$, and $\psi$ satisfies (H1)-(H3).
\begin{theorem}
\label{uniformmag}
For $\Psi$ given by \eqref{isoPsi} with the above hypotheses on $\alpha$ and $h$, the inequality \eqref{qc} holds whenever $M=\lambda R$, $\lambda>0$, $R\in SO(n)$.

Suppose further that the genericity condition 
\begin{align}
\label{generic}
&a+\lambda Q\om_1=\lambda \om_1,\;  c_2(a+\lambda Qx)=c_2(\lambda Rx)\text{ for }a\in\R^n,\,Q\in SO(n)\\
&\text{ implies }a=0 \text{ and }Q=R\nonumber
\end{align} 
holds. Then $y(x)=\lambda Rx$ is the unique minimiser of $I_{P_1,P_2}$ in $\mA_n$.
\end{theorem}
\begin{proof}
Let $y\in\mA_n$. By the arithmetic mean -- geometric mean inequality and the fact that  $\det Dy=\prod_{i=1}^nv_i$, we have that
\begin{eqnarray*}
\av_{\om_1}\Psi(Dy)\,\dx&\geq& \av_{\Omega_1}\left(n\left((\det Dy)^{\frac{\alpha}{n}}+(\det Dy)^{1-\frac{\alpha}{n}}\right)+h(\det Dy)\right)\,\dx\\
&=& \av_{\Omega_1}H(\det Dy(x))\,\dx\\
&\geq & H\left(\av_{\Omega_1} \det Dy(x)\,\dx\right)\\
&=&H(\lambda^n)= \Psi(\lambda{\bf 1}),
\end{eqnarray*}
as required, where we have set $H(\delta):=n(\delta^\frac{\alpha}{n}+\delta^{1-\frac{\alpha}{n}})+h(\delta)$ and used Jensen's inequality, noting that $H$ is convex and that $\int_{\om_1}\det Dy(x)\,\dx=\mL^n(y(\om_1))$.

We have equality in \eqref{qc} only when each $v_i=\lambda$, i.e. $Dy(x)=\lambda Q(x)$ for $Q(x)\in SO(n)$, which implies by \cite{reshetnyak67} that $Q(x)=Q$ is constant and thus $y(x)=a+\lambda Qx$  for some $a\in\R^n$. Thus $y(x)=\lambda x$ is the unique minimiser of $I_{P_1,P_2}$ provided \eqref{generic} holds.
\end{proof}

\subsection{General linearly related images}
\label{general}
In contrast to the case of uniform rotations and magnifications, for \eqref{qc} to hold for general $M$ implies that $\Psi$ has a very special form.

\begin{theorem}\label{generalM}
	Let $\Omega \subset \R^n$ be a bounded domain. If $n>1$ suppose that there exist $n$ points $x_i \in \partial\Omega$ in a neighbourhood of each of which $\partial\Omega$ is $C^1$, and for which the corresponding unit outward normals $\nu_i$  span $\R^n$. Let $\Psi\in C^0(M^{n\times n}_+)$ and 
$$\mA_M:=\{ y:\om\to\R^n: y \text{ an orientation-preserving diffeomorphism of } \om \text{ onto }M\om \}.$$ Then
\be
\label{qcM}
\av_\om\Psi(Dy)\,\dx\geq\Psi(M)
\ee
for all $M\in M^{n\times n}_+$ and all $y\in\mA_M$ if and only if
	\be
\label{convdet}
		\Psi(A) = h(\det A)
\ee
	for some convex $h: (0,\infty) \to \R$.
\end{theorem}
\begin{proof}
	We show that if \eqref{qcM} holds for all $M\in M^{n\times n}_+$ and $y\in\mA_M$ then $\Psi(M)=h(\det M)$ for some continuous $h:(0,\infty) \to \R$. This is immediate if $n=1$, so suppose that  $n>1$.

We first choose coordinates such that $0 \in \partial\Omega$,  $\partial\Omega$ is $C^1$ in some neighbourhood of $0$,  and  the unit outward normal at $0$ is $ e_n = (0,\dots,0,1)$. Then for some $\ep>0$
		\begin{align*}
			\Omega \cap B_\ep = \{x \in B_\ep \,:\, x_n < f(x')\},\\
			\partial\Omega \cap B_\ep = \{x \in B_\ep \,:\, x_n = f(x')\}
		\end{align*}
where $B_\ep=B(0,\ep)$, and
	 $f:\R^{n-1}\to\R$ is a $C^1$ function of $x' \equiv (x_1,x_2,\dots,x_{n-1})$ with $\nabla f(0)=0$. 
	
	Let $ p\in\R^n$ belong to the tangent space of $\partial\om$ at $0$, so that $p \cdot  e_n = 0$ and $p=(p',0)$. Let $\varphi\in C_0^\infty(\R^n)$  satisfy
		\begin{equation}\label{phibound}
			\max_{\R^n} \vert  p \cdot \nabla \varphi( x) \vert < 1.
		\end{equation}
	For $t>0$ sufficiently small and $x\in\R^n$ define
	\begin{equation}\label{variation}
		 x_t( x) =  x + t p \varphi\left(\frac{ x}{t} \right)
		+ \left[f\left( x' + t p'\varphi\left(\frac{ x}{t}\right)\right) - f( x')\right] e_n.
	\end{equation}
 Then $x_t:\R^n\to\R^n$ is $C^1$, and 
\begin{equation}\label{variationderivative}
	D x_t( x) = \1 +  p\otimes\nabla\varphi\left(\frac{ x}{t}\right) +  e_n \otimes \nabla_x F( x,t),
\end{equation}
where $F( x,t) := f\left( x' + t p'\varphi\left(\frac{ x}{t}\right)\right) - f( x')$.
We claim that $x_t:\R^n\to\R^n$ is a diffeomorphism with $ x_t(\Omega) = \Omega$, so that $x_t$ `slides at the boundary' in a small neighbourhood of $0$ in the direction of $p$. To this end we first show that
	\be
\label{limitis0}
\lim_{t \to 0} \max_{x\in\R^n} \vert \nabla_x F( x,t)\vert = 0.
\ee
	Note that $F(x,t) = 0$ for $x \notin t\,\supp \varphi$. We have that 
	\be
		\nabla_{x'}F(x,t) = \nabla f\left(x' +tp'\varphi\left(\frac x t\right)\right) - \nabla f(x') + \nabla f\left(x' +tp'\varphi\left(\frac x t\right)\right)\cdot p'\, \nabla_{x'}\varphi\left(\frac x t\right),
	\ee
	\be
		\frac{\partial F(x,t)}{\partial x_n} = 	\nabla f\left(x' +tp'\varphi\left(\frac x t\right)\right)\cdot p' \varphi_{,n}\left(\frac x t\right).
	\ee
If $x \in t\,\supp \varphi$, then $\vert x' + tp'\varphi\left(\frac x t\right)\vert \leq Ct$ for some $C > 0$, and since $f \in C^1(\R^{n-1})$ with $\nabla f(0) = 0$, we see that 
	\be
		\lim_{t\to 0}\sup_{\vert w \vert \leq Ct}\vert\nabla f (w)\vert = 0,
	\ee
and \eqref{limitis0} follows.

By \eqref{phibound} we have that
\begin{equation}
	\det\left(\1 +  p \otimes \nabla\varphi\left(\frac{ x}{t}\right)\right) = 1 +  p \cdot \nabla \varphi\left(\frac{ x}{t}\right) \geq \delta > 0
\end{equation}
for all $x,t$ and some $\delta$. Now, since 
\begin{equation}
	\det D  x_t( x) =
	\det\left(\1 +  p \otimes \nabla\varphi\left(\frac{ x}{t}\right)\right) \cdot \det(\1 + (\1 +  p \otimes \nabla\varphi\left(\frac{ x}{t}\right))^{-1})( e_n \otimes \nabla_x F( x,t)),
\end{equation}
by \eqref{limitis0}, $\det D  x_t( x) > 0$ for all $x$ if $t$ is sufficiently small. Since $ x_t( x) =  x$ for $\vert x\vert \geq R$ and some $R>0$, by Hadamard's global inverse function theorem (see, for example, \cite{gordon72})  $x_t:\R^n\to\R^n$ is a diffeomorphism. (An alternative argument is to note that $x_t$ is homotopic to the identity in $B_R$ so that for $p\in B_R$ the degree $d(x_t,B_R,p)=1$. Thus $x_t$ is invertible, and therefore a diffeomorphism by the local inverse function theorem.)
By \eqref{variation} 
\begin{gather}
	 x_t'( x) =  x' + t p'\varphi\left(\frac{ x}{t}\right)\\
	( x_t)_n( x) - f( x_t'( x)) = x_n - f( x'), 
\end{gather}
so that $(x',x_n)\in \om\cap B(0,\ep)$ iff $(x_t'(x),(x_t)_n(x))\in\om\cap B(0,\ep)$.  Hence $ x_t(\Omega) = \Omega$ as claimed. 

We now consider the variation of the deformation $ y( x) = M x$ given  by
\begin{equation}\label{yvariation}
	 y_t( x) :=M  x_t( x).
\end{equation}
Then $ y_t\in \mA_M$, and so from \eqref{qcM}, 
\begin{equation}
	\int_{\Omega} \left(\Psi(MD x_t( x)) - \Psi(M)\right) \mathrm{d} x \geq 0.
\end{equation}

Let $ x = t z$.  Denoting by $\chi_{t^{-1}\om}$ the characteristic function of $t^{-1}\om$ we have that
\begin{equation}\label{tineq}
	\int_{\mathrm{supp}\,\varphi}\chi_{t^{-1}\om}(z)
	\left(
	\Psi\left(M(\1 +  p \otimes \nabla \varphi( z)) 
	+  e_n \otimes \nabla_{ x}F(t z,t)\right)
	- \Psi(M)
	\right)
	\mathrm{d} z \geq 0.
\end{equation}
Suppose that $z_n < 0$. Then for $t$ sufficiently small, we have that  $t z \in B_\varepsilon$, and that $tz_n<f(tz)$ since $\lim_{t \to 0} t^{-1}f(t z) = 0$. Hence $t z \in \Omega$, and thus $\lim_{t\to 0} \chi_{t^{-1}\Omega}( z) = 1$. 
 Similarly, if $z_n > 0$ then $\lim_{t\to 0} \chi_{t^{-1}\Omega}(z) = 0$. So we can pass to the limit  $t \to 0$ in \eqref{tineq}  using bounded convergence and Lemma \ref{limitis0} to obtain
\begin{equation}\label{ineq1}
	\int_{\mathrm{supp}\,\varphi \cap \{z_n < 0\}}\left(
	\Psi(M(1 +  p \otimes \nabla \varphi( z))) - \Psi(M) \right){\mathrm d}z
	\geq 0.
\end{equation}
Now suppose that $\varphi \in W^{1,\infty}_0(B_R)$ for large $R$ and that $\esssup_ z \vert  p\cdot\nabla\varphi( z)\vert < 1 $. Let $\rho_\delta$ be a standard mollifier and define $\varphi_\delta:=\rho_\delta*\varphi$. Then for each $x\in\R^n$
\be
|p\cdot\nabla\varphi_\delta(x)|=\left|\int_{\R^n}\rho_\delta(x-z)p\cdot\nabla\varphi(z)\,{\mathrm d}z\right|<\int_{\R^n}\rho_\delta (x-z)\,{\mathrm d}z=1,
\ee
so that passing to the limit $\delta\to 0$ in \eqref{ineq1} for $\varphi_\delta$ we see that \eqref{ineq1} holds also for $\varphi$.

Now we choose 
\[
\varphi( x) := \left(1 - \vert x_n \vert\right)^+h(\vert  x' \vert ),
\]
where $\tau^+ := \max(0,\tau)$ denotes the positive part of $\tau$, and where  $h=
h(r)$ is the piecewise linear cut-off function defined by
$$h(r)=\left\{\begin{array}{ll}1, &|r|\leq L\\1+\frac{L+r}{d},&-(L+d)\leq r\leq -L\\1+\frac{L-r}{d},&L\leq r\leq L+d\\
0,& |r|\geq L+d.\end{array}\right.$$

Then 
\begin{equation}
	\nabla\varphi( x) = 
	\left(
	\nabla_{ x'}h(\vert x'\vert)\left(1 - \vert x_n\vert \right)^+,
	h(\vert x'\vert)\chi_{(-1,1)}(x_n)
	\right)
\end{equation}
so that
\begin{equation}
	\vert \nabla\varphi( x)\cdot  p \vert 
	\leq \vert\nabla_{ x'}h(\vert x'\vert)\vert\vert p\vert
	\leq
	\frac{\vert p\vert}{d} < 1
\end{equation}
for large enough $d$. Therefore from \eqref{ineq1} we have that
\begin{equation}
\begin{aligned}
	0 &\leq \int_{-1}^{0}\int_{\vert x'\vert < L + d}
	\left(
	\Psi(M(1 +  p \otimes \nabla\varphi)) - \Psi(M)
	\right)
	\dx'\,\dx_n\\
	&=
	\left(
	\Psi(M(\1 +  p \otimes  e_n)) - \Psi(M)
	\right)\mathcal{H}^{n-1}(B_L)
	+
	\int_{L+d \geq \vert x'\vert \geq L}
	\left(
	\Psi(M(\1 +  p \otimes \nabla\varphi)) - \Psi(M)
	\right)
	\dx'.
\end{aligned} 
\end{equation}
By construction, the last integral  is bounded up to a constant by $(L+d)^{n-1} - L^{n-1}$. Dividing by $L^{n-1}$ and letting $L \to \infty$, we see that 
\be\label{rank1inequality}
\Psi(M(\1 +  p \otimes  e_n)) \geq \Psi(M).
\ee
Now let $\bar{ x} \in \partial\Omega$ be such that $\partial\Omega$ is $C^1$ in a neighbourhood of $\bar{ x}$, and denote the unit outward normal at $\bar{ x}$ by $\nu$. Let $Q \in SO(n)$, $a \in \R^n$ with 
$
Q\bar{ x} + a = 0, \quad Q\nu =  e_n$,
and make the change of variables
\[
 z = Q x +  a, \quad \tilde{ y}( z) =  y( x)+MQ^Ta.
\]
Then
\[
D\tilde{ y}( z) = D y(x)Q^T, 
\]
so that by \eqref{qcM}, setting $\tilde{\Psi}(A) := \Psi(AQ)$, $\tilde M:=MQ^T$,
\be
	\int_{\Omega}
	(\Psi(Dy(x)) - \Psi(M))\,
	\dx 
	=  \int_{Q\Omega +  a}
	(\tilde{\Psi}(D\tilde{ y}( z)) - \tilde{\Psi}(\tilde M))\,
	{\mathrm d}z
	\geq 0 
\ee
whenever 
\be
\tilde{ y}(Q\Omega+ a) =  \tilde M(Q\Omega +  a).
\ee
Furthermore, $\partial(Q\om+a)$ is $C^1$ in a neighbourhood of $0$, and  $Q\om+a$ has unit outward normal $e_n$ there.
Hence by \eqref{rank1inequality} applied to $\tilde\Psi$ and $\tilde M$, we have that
\be
	\tilde{\Psi}(\tilde M(\1 +  \tilde p\otimes e_n)) \geq \tilde{\Psi}(\tilde M)
\ee
whenever $\tilde p\cdot e_n=0$, or equivalently, setting $p=Q\tilde p$,
\be
	\Psi(M(\1+ p\otimes\nu)) \geq \Psi(M)
\ee
whenever $ p\cdot \nu = 0$.

So, applying this argument to each of the $x_i$  and their corresponding normals  $\nu_i$, we have that
\be
\label{r1ineq}
\Psi(M(\1+ p_i\otimes\nu_i)) \geq \Psi(M)
\ee
for each $i$, whenever $ p_i\cdot\nu_i = 0$. We will use the following lemma, which  allows us to represent arbitrary invertible matrices as products of shears.

\begin{lemma}\label{lemma:rank1generation}
	Let $\nu_1,\ldots,\nu_n$ span $\R^n$. Then $SL(n,\R)=\{A\in M^{n\times n}:\det A=1\}$ is generated by finite products of  matrices of the form $\1 +  p_k\otimes\nu_{i_k}$ with $ p_k\cdot\nu_{i_k} = 0$.
\end{lemma}
\begin{proof}It is well known (see e.g. \cite[Theorem 3.2.10]{robinson82})  that
	$SL(n,\R)$ is generated by transvections, that is matrices of the form 
	\[
	\1 + \alpha  e_i\otimes e_j, \quad \alpha \in \R, \, i \neq j.
	\]
Since the $\nu_i$ span $\R^n$ there exists $A\in M^{n\times n}$, $\det A \neq 0$, with $\nu_i = A e_i$, $1 \leq i \leq n$. Given $C \in SL(n,\R)$ we have that $\det(A^{-T}CA^T) = 1$ and so
\be
\begin{aligned}
	C &= A^T \prod_{k=1}^{N}(\1 + \alpha_k  e_{i_k}\otimes e_{j_k})A^{-T}\\
&=\prod_{k=1}^NA^T(\1+\alpha_ke_{i_k}\otimes e_{j_k})A^{-T}\\
	&= \prod_{k=1}^{N}(\1 + \alpha_k A^T e_{i_k}\otimes A^{-1} e_{j_k})\\
	&= \prod_{k=1}^{N}(\1 + \alpha_k  p_k\otimes \nu_{j_k}),
\end{aligned}
\ee
and $ p_k\cdot\nu_{j_k} =  A^T e_{i_k}\cdot A^{-1} e_{j_k} = 0$.
\end{proof}
Now, let $M, M'\in M^{n\times n}_+$ have the  same determinant $\det M = \det M'$. By Lemma \ref{lemma:rank1generation} we can write 
\be
M^{-1}M' = \prod_{k=1}^{N}(1 +  p_k\otimes\nu_{i_k})
\ee
with $ p_k\cdot\nu_{i_k} = 0$. Hence, using \eqref{r1ineq} successively,
\be
\begin{aligned}
	\Psi(M') &= \Psi(MM^{-1}M')\\
			 &=\Psi\left(M\prod_{k=1}^{N}(1 +  p_k\otimes\nu_{i_k})\right)\\
			 &\geq \Psi\left(M\prod_{k=1}^{N-1}(1 +  p_k\otimes\nu_{i_k})\right)\\
			 &\geq \cdots \geq \Psi(M).
\end{aligned}
\ee
By a symmetric argument $\Psi(M) \geq \Psi(M')$ and hence $\Psi(M) = \Psi(M')$.
 Therefore 
\be\label{convexfcn}
\Psi(M) = h(\det M)
\ee
for some $h:(0,\infty)\to \R$, and 
$h$ is continuous (e.g. by choosing $M = \lambda \1$). 

To show that $h$ is convex for $n\geq 1$ we  modify the proof in \cite[Theorem 4.1]{j26} to handle the slight complication that the maps in $\mA_M$ are required to be diffeomorphisms. Let $B$ be an open ball contained in $\om_1$. Without loss of generality we suppose that $B=B(0,2)$. Let $0<p\leq q<\infty$, $\lambda\in (0,1)$, $m=\lambda p+(1-\lambda)q$, and consider the radial map
$$y(x)=\frac{r(R)}{R}x,\; x\in \R^n,$$
where $R=|x|$, $r(R)>0$ for $R>0$ and
$$r^n(R)=\left\{\begin{array}{ll}
pR^n,&0\leq R\leq\lambda^{1/n},\\
qR^n+\lambda(p-q),&\lambda^{1/n}\leq R\leq 1,\\
mR^n,&1\leq R<\infty.\end{array}\right.$$
Note that  $\det Dy(x)=r'\left(\frac{r(R)}{R}\right)^{n-1}=\frac{d(r^n)}{{\mathrm d}(R^n)}$. Let $s=R^n$ and for sufficiently small $\ep>0$ let $\rho_\ep=\rho_\ep(s)$ be a symmetric mollifier. Define 
$$r_\ep(R):=((\rho_\ep*r^n)(R^n))^{1/n}=\left(\int_\R\rho_\ep(R^n-s)r^n(s^{1/n}){\mathrm d}s\right)^{1/n}$$
 and $y_\ep(x):=\frac{r_\ep(R)}{R}x$.  Then $r_\ep(R)=p_{1/n}R$ for $0\leq R\leq\half\lambda^{1/n}$, $r_\ep(R)=m^{1/n}R$ for $R\geq 3/2$, and  $y_\ep$ is smooth with $\det Dy_\ep(x)>0$, so that (explicitly or by Hadamard's global inverse function theorem) $y_\ep\in\mA_{m^{1/n}\1}$. Also $Dy_\ep(x)$ is uniformly bounded, and $Dy_\ep(x)\to Dy(x)$ a.e. in $\om_1$, so that
\be
\label{convineq}
\lim_{\ep\to 0}\int_{\om_1}h(\det Dy_\ep(x))\,\dx=\int_{\om_1}h(\det Dy(x))\,\dx.
\ee
Hence from \eqref{qcM} we obtain
$$\av_{\om_1}h(\det Dy(x))\,\dx-h(m)=(\lambda h(p)+(1-\lambda)h(q)-h(m))\frac{\mL^n(B(0,1))}{\mL^n(\om_1)}\geq 0,$$
so that $h$ is convex.

Conversely, if \eqref{convexfcn} holds with $h$ convex then \eqref{qcM} follows by Jensen's inequality.
\end{proof}
\begin{remark}\rm We note that the hypothesis on $\om$ in Theorem \ref{generalM} is satisfied if \\
(i) if $\om$ is of class $C^1$, since in this case for any $\nu\in S^{n-1}$ there is a point $x\in\partial\om$ with outward normal $\nu$, \\ (ii) if $\om=\bigcap_{i=1}^N\{x\in\R^n:x\cdot n_i<\alpha_i\}$ is a convex polyhedron, since if the $n_i$ do not span $\R^n$ there exists a unit vector $m$ with $m\cdot n_i=0$ for all $i$, and hence $x+tm\in\om$ for all $x\in\om$ and $t\in\R$, contradicting the boundedness of $\om$.
\end{remark}

\subsection{The mass-preserving case}
\label{massp}
In the case when the intensity of magnified images diminishes in proportion to the Jacobian of the deformation gradient, a different question is appropriate, namely whether we can choose $\psi$ such that  if for some $M\in M^{n\times n}_+$ and $a\in\R^n$ we have
\be
\label{mag1}
\om_2=M\om_1+a,\;c_2(Mx+a)\det M=c_1(x),
\ee
then the unique minimiser $y$ of $I_{P_1,P_2}$ is $y(x)=Mx+a$. In this case it is appropriate to consider $\psi$ of the form \eqref{8aa}, \eqref{8b}, that is 
\be
\label{mass2}
\psi(c_1,c_2,A)=\Psi(A)+(1+(\det A)^{-1})|c_1-c_2\det A|^2,
\ee
where  $\Psi:M^{n\times n}_+\to[0,\infty)$ is isotropic with $\Psi(A)=\det A\cdot\Psi(A^{-1})$, $\Psi^{-1}(0)=SO(n)$. Again we can without loss of generality assume that $a=0$, so that we require that for all $y\in\mA_n$ with $y(\om_1)=M\om_1$ and $c_1,c_2$ with $c_2(Mx)\det M=c_1(x)$
\be
\label{newineq}
\hspace{.3in}\int_{\om_1}\left(\Psi(Dy(x))+(1+(\det Dy(x))^{-1})|c_1(x)-c_2(y(x))\det Dy(x)|^2\right)\,dx\geq\int_{\om_1}\Psi(M)\,dx,
\ee
with equality iff $y(x)=Mx$. Then as before \eqref{qc} is a sufficient condition for \eqref{newineq} to hold. It is also necessary provided we assume that $K=[0,1]^m$ and \eqref{newineq} holds for some sequence $c_1^{(j)}:\om_1\to K$, $c^{(j)}_2(z):= c_1^{(j)}(M^{-1}z)(\det M)^{-1}$ for all $j$ sufficiently large, with $c_1^{(j)}(x)\to 0$ boundedly a.e. in $\om_1$. Choosing $\Psi$ as in \eqref{isoPsi} we therefore have by Theorem \ref{uniformmag} that $y=\lambda Rx$ is a minimiser of $I_{P_1,P_2}$ whenever $M=\lambda R$ with $\lambda>0$, $R\in SO(n)$, and a minor modification of the proof shows that this minimiser is unique under the same nondegeneracy condition \eqref{generic}.

\subsection{Integrands depending on second derivatives}
Theorem \ref{generalM} implies in particular that for \eqref{Mcond} to hold, $\psi$ given by \eqref{psiform} cannot satisfy (H2), so that existence of a minimiser is not guaranteed. (Also $\Psi^{-1}(0)=SO(n)$ cannot hold.) However, by adding a dependence of the integrand on the second derivative $D^2y$ it is possible both to recover existence and retain the property that for linearly related images as in \eqref{lin} the unique minimiser is $y(x)=Mx$. To illustrate this we consider the functional
\be
\nonumber
E_{P_1,P_2}(y):=\int_{\om_1}\left(h(\det Dy)+(1+\det Dy)|c_1(x)-c_2(y(x))|^2+|D^2y(x)|^m\right)\,\dx&&\\&&\hspace{-1in}+\int_{\om_2}|D^2y^{-1}(z)|^m\,{\mathrm d}z,\label{newfunct}
\ee
where $h:(0,\infty)\to [0,\infty)$ is convex with $\delta h(\frac{1}{\delta})=h(\delta)$, $\lim_{\delta\to 0+}h(\delta)=\infty$, and where $m\geq \max(1,\frac{n}{2})$. (In particular we can take $m=2$ for $n\leq 4$.) Note that $E_{P_1,P_2}$ is symmetric with respect to interchanging images.

    As the set of admissible maps we take
\be
\label{Am2}
\mA_{2,m}=\{y\in W^{2,m}(\om_1,\R^n):y :\om_1\to\om_2\text{ an orientation-preserving}&&\\&&\hspace{-1.9in}\text{homeomorphism}, \;y^{-1}\in W^{2,m}(\om_2,\R^n)\}.\nonumber
\ee
A sufficient condition for $\mA_{2,m}$ to be nonempty is that there exists a smooth orientation preserving diffeomorphism between $\bar\om_1$ and $\bar\om_2$. Note that $m\geq\max(1,\frac{n}{2})$ implies that $W^{2,m}(\om_i,\R^n)$ is  continuously embedded in $W^{1,n}(\om_i,\R^n)$ for $i=1,2$.
\begin{theorem}
\label{Dtwo}
Suppose that $\rm (H4)$ and the above conditions on $h$ and $m$ hold, and that $\mA_{2,m}$ is nonempty. Then there exists an absolute minimiser $y^*$ of $E_{P_1,P_2}(y)$ in $\mA_{2,m}$. Furthermore, if $P_1,P_2$ are linearly related as in \eqref{lin} then $y^*(x)=Mx$ is an absolute minimiser, which is unique provided $P_1,P_2$ have no nontrivial affine symmetry in the sense that if $a\in\R^n$ and $A\in M^{n\times n}_+$ with $a+A\om_1=\om_2$ and $c_2(Mx)=c_2(a+Ax)$ a.e then $a=0$ and $A=M$ .
\end{theorem}
\begin{proof}
The proof of the existence of $y^*$ is an easy modification of that of Theorem \ref{exthm}, noting that minimizing sequences $y^{(j)}$ and their inverses $\xi^{(j)}$ are bounded in $W^{2,m}(\om_1,\R^n)$, $W^{2,m}(\om_2,\R^n)$ and hence in $W^{1,n}(\om_1,\R^n)$  and that the second derivative terms are convex.

If $P_1,P_2$ are linearly related as in \eqref{lin} and $y\in\mA_{2,m}$ then by Jensen's inequality and $\int_{\om_1}\det Dy(x)\,dx=\mL^n(\om_2)=(\det M)\mL^n(\om_1)$
\be
E_{P_1,P_2}(y)\geq \int_{\om_1}h(\det Dy(x))\,\dx\geq \mL^n(\om_1)\,h(\det M)=E_{P_1,P_2}(y^*),
\ee
where $y^*(x)=Mx$. Conversely, if $y^*$ is a minimiser then $D^2y^*(x)=0$ a.e. and hence $y^*(x)=a+Ax$ for some $A\in M^{n\times n}_+$ with $c_2(Mx)=c_2(a+Ax)$. Thus if  $P_1,P_2$ have no nontrivial affine symmetry $y^*(x)=Mx$.
\end{proof}

An analogous result holds in the mass-preserving case.
\section{Remarks on regularity and numerical implementation}
\label{regularitya}
In order to understand some of the issues relating to regularity and rigorous numerical implementation, it is worth recalling the state of the art for the case of pure nonlinear elasticity with no image comparison terms, namely for energy functionals of the form
\be
\label{nle1}
I(y)=\int_\om\Psi(Dy(x))\dx,
\ee
with $\Psi:M^{n\times n}_+\to [0,\infty)$ a sufficiently smooth free-energy function satisfying 
\be
\label{detcondition}\lim_{\det A\to 0+}\Psi(A)=\infty.
\ee
 Although one would like to be able to handle quasiconvex $\Psi$, the existence theory for minimisers of $I$ subject to appropriate boundary conditions is currently restricted to the case of polyconvex $\Psi$ satisfying growth hypotheses such as (H2), with models of natural polymers such as rubber having a form similar to \eqref{isoPsi}. Although a common failure mechanism for such polymers is cavitation (see \cite{j19}), growth conditions such as (H2) do not allow cavitation, and from an experimental perspective it seems reasonable to expect that minimisers are smooth. However this has not been proved for any interesting case if $n\geq 2$ (for a discussion see \cite[Section 2.3]{p31}).  It is not even known if minimisers are locally Lipschitz. Furthermore, despite \eqref{detcondition}, it is a longstanding open problem to decide whether every minimiser $y$ satisfies
\be
\label{unifdet}
\det Dy(x) \geq \mu>0 \text { for a.e. }x\in \om
\ee
for some $\mu$.

When we add image comparison terms the most we can hope for in general as regards the regularity of minimisers is that they are bi-Lipschitz. In fact  consider the case $n=m=1$ and $\om_1=\om_2=(0,1)$  with $\psi$ given by \eqref{8aa}, \eqref{8a} and $g(c_1,c_2)=\ep^{-1}|c_1-c_2|^2$, where we will later  choose $\ep>0$ sufficiently small. Thus we have to minimise
\be
\label{special}I(y)=\int_0^1\left(\ep\Psi(y_x)+(1+y_x)|c_1(x)-c_2(y(x))|^2\right)\,dx
\ee
over $y\in \mA_1:=\{y\in W^{1,1}(0,1): y(0)=0, y(1)=1, y_x>0 \text{ a.e.},\, y^{-1}\in W^{1,1}(0,1)\}$. We assume that $\Psi\geq 0$ is $C^1$ and strictly convex with $\Psi(1)=0$ and  $\lim_{p\to 0+}\Psi(p)=\lim_{p\to\infty}\frac{\Psi(p)}{p}=\infty$. Let us take
$$c_1(x)=\chi_{(0,\half)}(x),\;c_2(x)=\chi_{(0,\frac{3}{4})}(x).$$
Then by Theorem \ref{exthm} there exists a minimiser $y$ of $I$ in $\mA_1$ and $a:=y^{-1}(\frac{3}{4})\in(0,1)$. Suppose first that $a\leq \half$. Then $|c_1(x)-c_2(y(x))|^2=\chi_{(a,\half)}(x)$, so that $y$ is a minimiser of 
$$I(v)=\int_0^1\left(\ep\Psi(v_x)+(1+v_x)\chi_{(a,\half)}(x)\right)\,dx$$
for $v\in \mA_1$ such that $v(a)=y(a),\,v(\half)=y(\half)$. Since $\Psi$ is strictly convex  we have that $y_x$ is constant in each of the intervals $(0,a)$, $(a,\half)$ and $(\half,1)$, and if these constants are all equal then we must have $y(x)=x$. The same argument holds in the case $a\geq\half$. Hence the minimiser $y$ is Lipschitz but not $C^1$ unless $y(x)=x$, when $I(y)= \half$. But choosing 
$$\bar y(x)=\left\{\begin{array}{ll}\frac{3}{2}x,&0\leq x\leq\half,\\ \half(1+x),&\half\leq x\leq 1,
\end{array}\right.$$
 we find that $I(\bar y)=\half\ep\left(\Psi(\frac{3}{2})+\Psi(\half)\right)<\half$ for $\ep>0$ sufficiently small, so that the minimiser is bi-Lipschitz but not $C^1$.

Even in the special case \eqref{special} it does not seem obvious whether every minimiser is bi-Lipschitz for arbitrary $c_1, c_2\in L^\infty(\om)$. However if we assume more regularity, that  $c_1,c_2\in C^1([0,1])$, then any minimiser $y$  is a  $C^1$ diffeomorphism, so that in particular $y_x(x)\geq \mu>0$ for a.e. $x\in(0,1)$. This follows by writing (taking $\ep=1$)
$$I(y)=\int_0^1\left(\Psi(x,y_x)+\theta(x,y)\right)\,dx+\tau,$$
where $\Psi(x,p)=\Psi(p)+(1+p)|c_1(x)|^2$,  $\theta(x,y)=|c_2(y)|^2-2c_1(x)c_2(y)+2c_{1x}(x)C_2(y)$, $C_{2}(y):=\int_0^yc_2(s)\,ds$ and $\tau$ is a constant, and applying the regularity result \cite[Th\'eor\`eme 2]{p10} for integrands having this form.

As for nonlinear elasticity, from a rigorous perspective the formulation of a provably convergent minimization scheme is challenging  due to the possible occurence of the Lavrentiev phenomenon, whereby the infimum of the energy may depend on the function space, this being related to possible unboundedness of $Dy$ or vanishing of $\det Dy$. Such difficulties are compounded for our models due to possible discontinuities in the intensity maps.  Furthermore the minimization is to be carried out in a space of Sobolev homeomorphisms.  In particular, the approximation of Sobolev homeomorphisms by piecewise affine homeomorphisms, such as given by piecewise affine finite elements, is a tricky issue, that has only been resolved for $n=2$ by Iwaniec, Kovalev and Onninen \cite{IwaniecKovalevOnninen}. Some of the  methods for overcoming the Lavrentiev phenomenon that have been proposed and analyzed are described in \cite{p29}. 

These issues are to some extent simplified, but not eliminated, for energy functionals with regularizing higher derivative terms, for which higher order elements are required.  For such models, for example,  it is possible to prove the existence of minimizers satisfying \eqref{unifdet} under suitable growth conditions \cite{palmerhealey}. In second and higher order gradient theories of elasticity higher derivative terms are typically added to represent interfacial energy, often in cases for which the underlying nonlinear elastic energy is not quasiconvex, such as in the modelling of martensitic microstructures (see, for example, \cite{j32,j60}). It is possible that the use of such non quasiconvex energies could be of value for the analysis of images with fine structure, such as those arising in materials science.

\section*{Acknowledgements} We are grateful to Alexander  Belyaev, Jos\'e Iglesias, Tadeusz Iwaniec, David Mumford, Ozan \"{O}ktem, Jani Onninen, Martin Rumpf, Carola Schonlieb, Gabriele Steidl, L\'aszl\'o Sz\'ekelyhidi and  Benedikt Wirth for their interest and helpful suggestions. We are also very grateful to the referees for their careful reading of the paper which led to numerous corrections and improvements. The paper was substantially revised while JMB visited the Hong Kong Institute for Advanced Study. CLH was supported by EPSRC through grant EP/L016508/1. 

 \bibliographystyle{plain}
 
\bibliography{balljourn,ballconfproc,gen2,computervision}

\begin{thebibliography}{10}

\bibitem{astalaetal2023}
K.~Astala, D.~Faraco, A.~Guerra, A.~Koski, and J.~Kristensen.
\newblock The local {B}urkholder functional, quasiconvexity and geometric
  function theory.
\newblock {\em arXiv preprint arXiv:2309.03495}, 2023.

\bibitem{p4}
J.~M. Ball.
\newblock Constitutive inequalities and existence theorems in nonlinear
  elastostatics.
\newblock In R.~J. Knops, editor, {\em Nonlinear Analysis and Mechanics,
  Heriot-Watt Symposium, Vol. 1}. Pitman, 1977.

\bibitem{j8}
J.~M. Ball.
\newblock Convexity conditions and existence theorems in nonlinear elasticity.
\newblock {\em Arch. Ration. Mech. Anal.}, 63:337--403, 1977.

\bibitem{j16}
J.~M. Ball.
\newblock Global invertibility of {S}obolev functions and the interpenetration
  of matter.
\newblock {\em Proc. Royal Soc. Edinburgh}, 88A:315--328, 1981.

\bibitem{p10}
J.~M. Ball.
\newblock Remarques sur l'existence et la r\'egularit\'e des solutions
  d'\'elastostatique non~lin\'eaire.
\newblock In H.~Berestycki and H.~Brezis, editors, {\em Recent Contributions to
  Nonlinear Partial Differential Equations}. Pitman, 1981.

\bibitem{j19}
J.~M. Ball.
\newblock Discontinuous equilibrium solutions and cavitation in nonlinear
  elasticity.
\newblock {\em Phil. Trans. Royal Soc. London A}, 306:557--611, 1982.

\bibitem{p29}
J.~M. Ball.
\newblock Singularities and computation of minimizers for variational problems.
\newblock In R.~DeVore, A.~Iserles, and E.~Suli, editors, {\em Foundations of
  Computational Mathematics}, volume 284 of {\em London Mathematical Society
  Lecture Note Series}, pages 1--20. Cambridge University Press, 2001.

\bibitem{p31}
J.~M. Ball.
\newblock Some open problems in elasticity.
\newblock In {\em Geometry, Mechanics, and Dynamics}, pages 3--59. Springer,
  New York, 2002.

\bibitem{j60}
J.~M. Ball and E.~C.~M. Crooks.
\newblock Local minimizers and planar interfaces in a phase-transition model
  with interfacial energy.
\newblock {\em Calc. Var. Partial Differential Equations}, 40(3-4):501--538,
  2011.

\bibitem{j17}
J.~M. Ball, J.~C. Currie, and P.~J. Olver.
\newblock Null {L}agrangians, weak continuity, and variational problems of
  arbitrary order.
\newblock {\em J. Functional Anal.}, 41:135--174, 1981.

\bibitem{ballhorner}
J.~M. Ball and C.~L. Horner.
\newblock Image comparison and scaling via nonlinear elasticity.
\newblock In L.~Calatroni, M.~Donatelli, S.~Morigi, M.~Prato, and
  M.~Santacesaria, editors, {\em Scale Space and Variational Methods in
  Computer Vision}, pages 565--574, Cham, 2023. Springer International
  Publishing.

\bibitem{j32}
J.~M. Ball and R.~D. James.
\newblock Fine phase mixtures as minimizers of energy.
\newblock {\em Arch. Ration. Mech. Anal.}, 100:13--52, 1987.

\bibitem{j26}
J.~M. Ball and F.~Murat.
\newblock {$W^{1,p}$}-quasiconvexity and variational problems for multiple
  integrals.
\newblock {\em J. Functional Analysis}, 58:225--253, 1984.

\bibitem{ballonnineniwaniec}
J.M. Ball, J.~Onninen, and T.~Iwaniec.
\newblock Equicontinuity up to the boundary of weakly convergent sequences of
  {S}obolev homeomorphisms.
\newblock In preparation.

\bibitem{banyaga}
A.~Banyaga.
\newblock {\em The structure of classical diffeomorphism groups}, volume 400 of
  {\em Mathematics and its Applications}.
\newblock Kluwer Academic Publishers Group, Dordrecht, 1997.

\bibitem{BerkelsEfflandRumpf2015}
B.~Berkels, A.~Effland, and M.~Rumpf.
\newblock Time discrete geodesic paths in the space of images.
\newblock {\em SIAM Journal on Imaging Sciences}, 8(3):1457--1488, 2015.

\bibitem{brezis2011}
H.~Brezis.
\newblock {\em Functional analysis, {S}obolev spaces and partial differential
  equations}.
\newblock Universitext. Springer, New York, 2011.

\bibitem{burgeretal2013}
M.~Burger, J.~Modersitzki, and L.~Ruthotto.
\newblock A hyperelastic regularization energy for image registration.
\newblock {\em SIAM J. Sci. Comput.}, 35(1):B132--B148, 2013.

\bibitem{cachierrey2000}
P.~Cachier and D.~Rey.
\newblock Symmetrization of the non-rigid registration problem using
  inversion-invariant energies: Application to multiple sclerosis.
\newblock In {\em Medical Image Computing and Computer-Assisted
  Intervention--MICCAI 2000: Third International Conference, Pittsburgh, PA,
  USA, October 11-14, 2000. Proceedings 3}, pages 472--481. Springer, 2000.

\bibitem{debrouxetal20}
N.~Debroux, J.~Aston, F.~Bonardi, A.~Forbes, C.~Le~Guyader, M.~Romanchikova,
  and C.-B. Schonlieb.
\newblock A variational model dedicated to joint segmentation, registration,
  and atlas generation for shape analysis.
\newblock {\em SIAM Journal on Imaging Sciences}, 13(1):351--380, 2020.

\bibitem{debroux}
N.~Debroux and C.~Le~Guyader.
\newblock A joint segmentation/registration model based on a nonlocal
  characterization of weighted total variation and nonlocal shape descriptors.
\newblock {\em SIAM J. Imaging Sci.}, 11(2):957--990, 2018.

\bibitem{DellacherieMeyer}
C.~Dellacherie and P.-A. Meyer.
\newblock {\em Probabilities and potential}, volume~29 of {\em North-Holland
  Mathematics Studies}.
\newblock North-Holland Publishing Co., Amsterdam-New York; North-Holland
  Publishing Co., Amsterdam-New York, 1978.

\bibitem{droskerumpf04}
M.~Droske and M.~Rumpf.
\newblock A variational approach to non-rigid morphological registration.
\newblock {\em SIAM Journal on Applied Mathematics}, 64(2):668--687, 2004.

\bibitem{EfflandNeumayerRumpf2020}
A.~Effland, S.~Neumayer, and M.~Rumpf.
\newblock Convergence of the time discrete metamorphosis model on {H}adamard
  manifolds.
\newblock {\em SIAM Journal on Imaging Sciences}, 13(2):557--588, 2020.

\bibitem{eisen}
G.~Eisen.
\newblock A selection lemma for sequences of measurable sets, and lower
  semicontinuity of multiple integrals.
\newblock {\em Manuscripta Math.}, 27(1):73--79, 1979.

\bibitem{fonsecagangbo95a}
I.~Fonseca and W.~Gangbo.
\newblock {\em Degree theory in analysis and applications}, volume~2 of {\em
  Oxford Lecture Series in Mathematics and its Applications}.
\newblock The Clarendon Press, Oxford University Press, New York, 1995.
\newblock Oxford Science Publications.

\bibitem{gordon72}
W.~B. Gordon.
\newblock On the diffeomorphisms of {E}uclidean space.
\newblock {\em Amer. Math. Monthly}, 79:755--759, 1972.

\bibitem{guastinietal23}
M.~Guastini, M.~Rajkovi{\'{c}}, M.~Rumpf, and B.~Wirth.
\newblock The variational approach to the flow of {S}obolev-diffeomorphisms
  model.
\newblock In L.~Calatroni, M.~Donatelli, S.~Morigi, M.~Prato, and
  M.~Santacesaria, editors, {\em Scale Space and Variational Methods in
  Computer Vision}, pages 551--564, Cham, 2023. Springer International
  Publishing.

\bibitem{iglesias21}
J.~A. Iglesias.
\newblock Symmetry and scaling limits for matching of implicit surfaces based
  on thin shell energies.
\newblock {\em ESAIM: Mathematical Modelling and Numerical Analysis},
  55(3):1133--1161, 2021.

\bibitem{iglesiasrumpfscherzer}
J.~A. Iglesias, M.~Rumpf, and O.~Scherzer.
\newblock Shape-aware matching of implicit surfaces based on thin shell
  energies.
\newblock {\em Foundations of Computational Mathematics}, 18:891--927, 2018.

\bibitem{IwaniecKovalevOnninen}
T.~Iwaniec, L.~V. Kovalev, and J.~Onninen.
\newblock Diffeomorphic approximation of {S}obolev homeomorphisms.
\newblock {\em Arch. Ration. Mech. Anal.}, 201(3):1047--1067, 2011.

\bibitem{iwanieconninen2009}
T.~Iwaniec and J.~Onninen.
\newblock Hyperelastic deformations of smallest total energy.
\newblock {\em Arch. Ration. Mech. Anal.}, 194(3):927--986, 2009.

\bibitem{iwanieconninen2011a}
T.~Iwaniec and J.~Onninen.
\newblock Deformations of finite conformal energy: {B}oundary behavior and
  limit theorems.
\newblock {\em Trans. Amer. Math. Soc.}, 363(11):5605--5648, 2011.

\bibitem{kolourietal2015}
S.~Kolouri, D.~Slepcev, and G.~K. Rohde.
\newblock A symmetric deformation-based similarity measure for shape analysis.
\newblock In {\em 2015 IEEE 12th International Symposium on Biomedical Imaging
  (ISBI)}, pages 314--318. IEEE, 2015.

\bibitem{KrieglMichor}
A.~Kriegl and P.~W. Michor.
\newblock {\em The convenient setting of global analysis}, volume~53 of {\em
  Mathematical Surveys and Monographs}.
\newblock American Mathematical Society, Providence, RI, 1997.

\bibitem{linetal2010}
T.~Lin, I.~Dinov, A.~Toga, and L.~Vese.
\newblock Nonlinear elasticity registration and {S}obolev gradients.
\newblock In {\em International Workshop on Biomedical Image Registration},
  pages 269--280. Springer, 2010.

\bibitem{michor94}
P.~W. Michor and C.~Vizman.
\newblock {$n$}-transitivity of certain diffeomorphism groups.
\newblock {\em Acta Math. Univ. Comenian. (N.S.)}, 63(2):221--225, 1994.

\bibitem{MillerYounes}
M.~I. Miller and L.~Younes.
\newblock Group actions, homeomorphisms, and matching: A general framework.
\newblock {\em International Journal of Computer Vision}, 41:61--84, 2001.

\bibitem{modersitzki}
J.~Modersitzki.
\newblock {\em Numerical methods for image registration}.
\newblock Numerical Mathematics and Scientific Computation. Oxford University
  Press, New York, 2004.
\newblock Oxford Science Publications.

\bibitem{NeumayerPerschSteidl2018}
S.~Neumayer, J.~Persch, and G.~Steidl.
\newblock Morphing of manifold-valued images inspired by discrete geodesics in
  image spaces.
\newblock {\em SIAM Journal on Imaging Sciences}, 11(3):1898--1930, 2018.

\bibitem{OzereLeguyader2015}
S.~Ozer\'{e}, C.~Gout, and C.~Le~Guyader.
\newblock Joint segmentation/registration model by shape alignment via weighted
  total variation minimization and nonlinear elasticity.
\newblock {\em SIAM J. Imaging Sci.}, 8(3):1981--2020, 2015.

\bibitem{OzereLeguyader2015a}
S.~Ozer\'{e} and C.~Le~Guyader.
\newblock Topology preservation for image-registration-related deformation
  fields.
\newblock {\em Commun. Math. Sci.}, 13(5):1135--1161, 2015.

\bibitem{palmerhealey}
A.~Z. Palmer and T.~J. Healey.
\newblock Injectivity and self-contact in second-gradient nonlinear elasticity.
\newblock {\em Calc. Var. Partial Differential Equations}, 56(4):Paper No. 114,
  11, 2017.

\bibitem{reshetnyak67}
Y.~G. Reshetnyak.
\newblock Liouville's theorem on conformal mappings under minimal regularity
  assumptions.
\newblock {\em Siberian Math. J.}, 8:631--653, 1967.

\bibitem{reshetnyak67a}
Yu.~G. Reshetnyak.
\newblock Stability of conformal mappings in multi-dimensional spaces.
\newblock {\em Sibirsk. Mat. \v Z.}, 8:91--114, 1967.

\bibitem{reshetnyak68}
Yu.~G. Reshetnyak.
\newblock Stability theorems for mappings with bounded distortion.
\newblock {\em Sibirsk. Mat. \v Z.}, 9:667--684, 1968.

\bibitem{rindler2018}
F.~Rindler.
\newblock {\em Calculus of variations}.
\newblock Universitext. Springer, Cham, 2018.

\bibitem{robinson82}
D.~J.~S. Robinson.
\newblock {\em A course in the theory of groups}, volume~80 of {\em Graduate
  Texts in Mathematics}.
\newblock Springer-Verlag, New York-Berlin, 1982.

\bibitem{rumpf2013}
M.~Rumpf.
\newblock Variational methods in image matching and motion extraction.
\newblock In {\em Level Set and PDE Based Reconstruction Methods in Imaging},
  pages 143--204. Springer, 2013.

\bibitem{rumpfwirth}
M.~Rumpf and B.~Wirth.
\newblock An elasticity-based covariance analysis of shapes.
\newblock {\em Int. J. Comput. Vis.}, 92(3):281--295, 2011.

\bibitem{simonetal2017}
K.~Simon, S.~Sheorey, D.~W. Jacobs, and R.~Basri.
\newblock A hyperelastic two-scale optimization model for shape matching.
\newblock {\em SIAM J. Sci. Comput.}, 39(1):B165--B189, 2017.

\bibitem{smale59}
S.~Smale.
\newblock Diffeomorphisms of the {$2$}-sphere.
\newblock {\em Proc. Amer. Math. Soc.}, 10:621--626, 1959.

\bibitem{solomon2015convolutional}
J.~Solomon, F.~De~Goes, G.~Peyr{\'e}, M.~Cuturi, A.~Butscher, A.~Nguyen, T.~Du,
  and L.~Guibas.
\newblock Convolutional wasserstein distances: Efficient optimal transportation
  on geometric domains.
\newblock {\em ACM Transactions on Graphics (ToG)}, 34(4):1--11, 2015.

\bibitem{thurston}
W.~P. Thurston.
\newblock {\em Three-dimensional geometry and topology. {V}ol. 1}, volume~35 of
  {\em Princeton Mathematical Series}.
\newblock Princeton University Press, Princeton, NJ, 1997.

\bibitem{TrouveYounes}
A.~Trouv\'e and L.~Younes.
\newblock Local geometry of deformable templates.
\newblock {\em SIAM J. Math. Anal.}, 37(1):17--59, 2005.

\bibitem{TrouveYounesa}
A.~Trouv\'e and L.~Younes.
\newblock Metamorphoses through {L}ie group action.
\newblock {\em Found. Comput. Math.}, 5(2):173--198, 2005.

\bibitem{vodopyanov79}
S.~K. Vodop'yanov, V.~M. Gol'dshtein, and Yu.~G. Reshetnyak.
\newblock The geometric properties of functions with generalized first
  derivatives.
\newblock {\em Russian Math. Surveys}, 34:19--74, 1979.

\bibitem{younes}
L.~Younes.
\newblock {\em Shapes and diffeomorphisms}, volume 171 of {\em Applied
  Mathematical Sciences}.
\newblock Springer, Berlin, second edition, 2019.

\end{thebibliography}

\end{document}